\documentclass{article}
\usepackage[latin1]{inputenc}
\usepackage[T1]{fontenc}
\usepackage{amsmath,amssymb}
\usepackage{amsfonts,amsthm}
\usepackage{geometry}
\usepackage{color}
   \definecolor{cites}{rgb}{0.75 , 0.00 , 0.00}  
   \definecolor{urls} {rgb}{0.00 , 0.00 , 1.00}  
   \definecolor{links}{rgb}{0.00 , 0.00 , 0.5}   
  \definecolor{gray}{rgb}{0.5,0.5,.5}
\usepackage{graphicx}
\usepackage{cancel}
\usepackage[plainpages=false,colorlinks=true, citecolor=cites, urlcolor=urls, linkcolor=links, breaklinks=true]{hyperref}

\newcommand{\Cb}{\mathbb{C}}

\newcommand{\Nb}{\mathbb{N}}

\newcommand{\Rb}{\mathbb{R}}

\newcommand{\Zb}{\mathbb{Z}}

\newcommand{\Xb}{\textbf{\upshape X}}
\newcommand{\Yb}{\textbf{\upshape Y}}

\newcommand{\Ac}{\mathcal{A}}

\newcommand{\Cc}{\mathcal{C}}

\newcommand{\Kc}{\mathcal{K}}
\newcommand{\Lc}{\mathcal{L}}

\newcommand{\Pc}{\mathcal{P}}

\newcommand{\Alp}{\mathcal{A}(\Xb)}
\newcommand{\Aln}{\mathcal{A}_0(\Xb)}
\newcommand{\Alpr}{\mathcal{A}_{\$}(\Xb)}

\newcommand{\eps}{\varepsilon}
\newcommand{\ph}{\varphi}

\newcommand{\vertiii}[1]{{\left\vert\kern-0.25ex\left\vert\kern-0.25ex\left\vert #1
    \right\vert\kern-0.25ex\right\vert\kern-0.25ex\right\vert}}

\newcommand\pto{
   \unitlength0.1ex
   \begin{picture}(30,15)
   \put(15,16){\makebox(0,0)[]{\tiny $\Pc$}}
   \put(15,5){\makebox(0,0)[]{$\to$}}
   \end{picture}
}

\DeclareMathOperator{\coker}{coker}
\DeclareMathOperator{\diag}{diag}
\DeclareMathOperator{\diam}{diam}
\DeclareMathOperator{\dist}{dist}
\DeclareMathOperator{\im}{im}
\DeclareMathOperator{\ind}{ind}
\DeclareMathOperator{\ess}{ess}
\DeclareMathOperator{\plim}{\Pc-lim}

\DeclareMathOperator{\rk}{rank}

\DeclareMathOperator{\plimn}{\underset{\textit{n}\to\infty}{\Pc-lim \,}}
\DeclareMathOperator{\spc}{sp}
\DeclareMathOperator{\spess}{\spc_{\ess}}
\DeclareMathOperator{\spe}{sp_\eps}
\DeclareMathOperator{\speess}{\spc_{\eps,\ess}}
\DeclareMathOperator{\spn}{span}
\DeclareMathOperator{\supp}{supp}
\DeclareMathOperator{\nuess}{\nu_{ess}}
\DeclareMathOperator{\nuessc}{\nu_{ess,c}}
\DeclareMathOperator{\opsp}{\sigma_{op}}

\providecommand{\lb}[1]{\Lc(#1)}
\providecommand{\lc}[1]{\Kc(#1)}
\providecommand{\pb}[1]{\Lc(#1,\Pc)}
\providecommand{\pbr}[1]{\Lc_{\$}(#1,\Pc)}
\providecommand{\pc}[1]{\Kc(#1,\Pc)}
\providecommand{\pcn}[1]{\Kc_0(#1,\Pc)}

\newtheorem{thm}{Theorem}[section]

\newtheorem{lem}[thm]{Lemma}
\newtheorem{prop}[thm]{Proposition}
\newtheorem{cor}[thm]{Corollary}

\theoremstyle{definition}
\newtheorem{defn}[thm]{Definition}
\newtheorem{rem}[thm]{Remark}
\newtheorem{ex}[thm]{Example}

\numberwithin{equation}{section}

\begin{document}

\title{\bf Essential pseudospectra and essential norms\\of band-dominated operators}
\author{
Raffael Hagger\footnote{\href{mailto:Raffael.Hagger@tuhh.de}{\tt Raffael.Hagger@tuhh.de}, Hamburg University of Technology (TUHH), Germany},
Marko Lindner\footnote{\href{mailto:Marko.Lindner@tuhh.de}{\tt Marko.Lindner@tuhh.de}, Hamburg University of Technology (TUHH), Germany (corresponding author)}, 
Markus Seidel\footnote{\href{mailto:Markus.Seidel@fh-zwickau.de}{\tt Markus.Seidel@fh-zwickau.de}, University of Applied Sciences Zwickau, Germany}}
\maketitle
\vspace{-0.4cm}
\begin{abstract}
An operator $A$ on an $l^{p}$-space is called band-dominated if it can be approximated,
in the operator norm, by operators with a banded matrix representation. The coset of $A$
in the Calkin algebra determines, for example, the Fredholmness of $A$, the Fredholm index,
the essential spectrum, the essential norm and the so-called essential pseudospectrum of $A$.
This coset can be identified with the collection of all so-called limit operators of $A$.
It is known that this identification preserves invertibility (hence spectra). We now show
that it also preserves norms and in particular resolvent norms (hence pseudospectra).
In fact we work with a generalization of the ideal of compact operators, so-called
$\Pc$-compact operators, allowing for a more flexible framework that naturally
extends to $l^p$-spaces with $p\in\{1,\infty\}$ and/or vector-valued $l^p$-spaces.
\medskip
\textbf{AMS subject classification:}  47A53; 47B07, 46E40, 47B36, 47L80, 65J10.

\medskip
\textbf{Keywords:} Fredholm theory; Essential spectrum; Pseudospectra; Limit operator; Band-dominated operator.
\end{abstract}

\

\section{Introduction}
This first section comes as a rough guide to this paper. Proper definitions and
theorems are given in later sections.
\medskip

We study bounded linear operators on a Banach space $\Xb$. Most of the time,
$\Xb$ is an $l^p$ sequence space with $1\leq p\leq\infty$, index set $\Zb^N$ and values in another
Banach space $X$, so that an operator on $\Xb=l^p(\Zb^N,X)$ can be identified,
in a natural way, with an infinite matrix $(a_{ij})$ with indices $i,j\in\Zb^N$
and all $a_{ij}$ being operators $X\to X$.

For such an operator $A$ on $\Xb$, write $A\in\pcn{\Xb}$ if its matrix $(a_{ij})$
has finite support (i.e. only finitely many nonzero entries), and write
$A\in\Aln$ if its matrix is a band matrix (i.e. it has only finitely many
nonzero diagonals). Clearly, $\Aln$ is an algebra containing $\pcn{\Xb}$ as a (two-sided)
ideal. Denote the closure, in the $\Xb\to\Xb$ operator norm, of $\Aln$ by $\Alp$
and the closure of $\pcn{\Xb}$ by $\pc{\Xb}$. Then $\Alp$ is a Banach algebra
containing $\pc{\Xb}$ as a closed ideal\footnote{We will explain the notation
$\pc{\Xb}$ later and say what $\Pc$ is.}.

Operators in $\Alp$ are called band-dominated operators. The ideal $\pc{\Xb}$
is a generalization
of the set of compact operators: If $\dim X<\infty$ then
$\pc{\Xb}$ coincides with the set $\lc{\Xb}$ of all compact operators on $\Xb$
(except in the somewhat pathological cases $p=1$ and $p=\infty$);
otherwise it does not -- as already $\pcn{\Xb}$ contains non-compact operators.
Recall that $\lc{\Xb}$ is a closed ideal in the algebra $\lb{\Xb}$ of all bounded
linear operators $\Xb\to\Xb$.

For $A\in\Alp$, the coset
\begin{equation} \label{eq:coset}
A+\pc{\Xb} \textrm{ in the quotient algebra } \Alp/\pc{\Xb}
\end{equation}
is of interest. If $\pc{\Xb}=\lc{\Xb}$ then the quotient norm of \eqref{eq:coset}
is the usual essential norm of $A$, the spectrum of \eqref{eq:coset} is the
essential spectrum of $A$, and the invertibility of \eqref{eq:coset} corresponds
to $A$ being a Fredholm operator (i.e. having a finite-dimensional kernel and a
finite-codimensional range). In the general case one gets generalized
versions of these quantities and properties.

In either case, the coset \eqref{eq:coset} is an interesting but complicated
object. Our strategy for its study is a localization technique that replaces
this one complicated object by a family of many simpler objects.
The key observation is that, by the definition of the ideal $\pc{\Xb}$,
the coset \eqref{eq:coset} depends only (and exactly) on the asymptotic
behaviour of the matrix behind $A$. This asymptotic behaviour is extracted as follows:
For every $k\in\Zb^N$, let $V_k:\Xb\to\Xb$ denote the $k$-shift operator
that maps $(x_i)_{i\in\Zb^N}$ to $(y_{i})_{i\in\Zb^N}$ with $y_{i+k}=x_i$,
and then look at the translates $V_{-k}AV_k$ of $A$. The simpler objects
that characterize the coset \eqref{eq:coset} are the partial limits of the family
$(V_{-k}AV_k)_{k\in\Zb^N}$ of all translates of $A$ with respect to the
so-called $\Pc$-topology, to be described below,
that corresponds to entry-wise norm convergence of the matrix.
More precisely, if $h=(h_1,h_2,...)$ is a sequence in $\Zb^N$ with
$|h_n|\to\infty$ and $V_{-h_n}AV_{h_n}$ converges in that topology then we
denote the limit by $A_h$ and call it the limit operator of $A$ with respect
to the sequence $h$. Doing this with all such sequences that produce a limit
operator yields the collection
\begin{equation}\label{eq:opspec}
\opsp(A):=\{A_h\,:\, h=(h_1,h_2,...),\, h_n\in\Zb^N,\, |h_n|\to\infty,\, 
A_h:=\plim V_{-h_n}AV_{h_n} \textrm{ exists}\}
\end{equation}
of all limit operators -- the so-called operator spectrum of $A$. We have used
sequences $h$ to address our partial limits of $(V_{-k}AV_k)_{k\in\Zb^N}$.
The same set \eqref{eq:opspec} can also be constructed as follows (\cite{Roe},\cite{SpakWillett}):
Extend the mapping
$\ph_A:k\in\Zb^N\mapsto V_{-k}AV_k\in\Alp$ $\Pc$-continuously to the  (Stone-\v{C}ech)
boundary $\partial\Zb^N$ of $\Zb^N$. Then \eqref{eq:opspec}  exactly collects
the values of $\ph_A$ on $\partial\Zb^N$.
Enumerating the set \eqref{eq:opspec} via $\partial\Zb^N$ (rather than via the
set of all sequences $h$ in $\Zb^N$ for which $A_h$ exists) has the benefit that the
index set $\partial\Zb^N$ is independent of $A$, so that
two instances of \eqref{eq:opspec} can be added or multiplied elementwise. Under these operations,
the map $A\mapsto\ph_A|_{\partial\Zb^N}=\eqref{eq:opspec}$ turns out to be an
algebra homomorphism. Now the crucial point is that $\pc{\Xb}$ is exactly the
kernel of that homomorphism $A\mapsto \eqref{eq:opspec}$, whence
$\eqref{eq:coset}\mapsto\eqref{eq:opspec}$ is a well-defined algebra
isomorphism\footnote{To oversimplify matters, think of continuous functions $f$
on a compact set $D$. Then the subspace (actually the ideal) $C_0(D)$ of
continuous functions with zero boundary values is the kernel of the algebra
homomorphism $f\mapsto f|_{\partial D}$, whence the coset of $f$ modulo
$C_0(D)$ can be  identified with $f|_{\partial D}$, by the fundamental
homomorphism theorem.}. In short: The set \eqref{eq:opspec} nicely reflects
the coset \eqref{eq:coset}.
Actually, besides $A\in\Alp$, there is one technical condition to make
this identification
between the coset \eqref{eq:coset} and the set \eqref{eq:opspec} work: To make sure that
\eqref{eq:opspec} is large enough,
we have to assume that $\{V_{-k}AV_k:k\in\Zb^N\}$ has a sequential compactness property,
namely that every sequence $h$ in $\Zb^N$ with $|h_n|\to\infty$ has a subsequence $g$
for which the $\Pc$-limit $A_g$ exists, in which case we call $A$ a rich operator (in
the sense that \eqref{eq:opspec} is rich enough to reflect all\footnote{In fact, the map
$A\mapsto \opsp(A)=\eqref{eq:opspec}$ sends some operators $A\in\Alp$ to $\varnothing$.
For some other $A\in\Alp$, limit operators exist in one ``direction'' but not in another.
Some of the latter $A$ are not in $\pc{\Xb}$ but have $\opsp(A)=\{0\}$, such as our first
example in Remark \ref{RCounterEx}.
These problems are eliminated by imposing existence of sufficiently many limit operators,
i.e. richness of $A$.} of \eqref{eq:coset}).
\newpage  

This identification between the objects \eqref{eq:coset} and \eqref{eq:opspec} for a
rich operator $A\in\Alp$ is at the core of the limit operator method.
Here are some of its consequences:
\begin{enumerate}
\item[(i)] The main theorem on limit operators
\cite{RochFredTh,LimOps,Marko,BigQuest} says that \eqref{eq:coset} is invertible
(so that $A$ is a generalized Fredholm operator) iff every element of \eqref{eq:opspec}
is invertible.
\item[(ii)] Expressing this in the language of spectra, we get that
\begin{equation}\label{eq:specess}
\spess(A) = \bigcup_{A_h\in\opsp(A)} \spc(A_h),
\end{equation}
where $\spess(A)$ denotes the spectrum of the coset \eqref{eq:coset}, the so-called
$\Pc$-essential spectrum of $A$, and $\spc(A_h)$ denotes the usual spectrum of $A_h$
as an element of $\lb{\Xb}$.
\item[(iii)] In addition to (i), the inverse\footnote{\label{foot:invclosed}Remarkably, this inverse coset, 
resp. set of inverses, is again a subset of $\Alp$, by \cite[Propositions 2.1.8 and 2.1.9]{LimOps} 
(we recall this in Theorem \ref{TBDOInvCl} below).} of \eqref{eq:coset} corresponds to the elementwise
inverse$^{\ref{foot:invclosed}}$ of \eqref{eq:opspec}, by \cite[Theorem 16]{SeFre}.
\item[(iv)] In the current paper we show that the norm of \eqref{eq:coset}
equals\footnote{\label{foot:C*alg}Note that in the case $\Xb=l^2(\Zb^N,X)$ with a Hilbert space $X$,
$\lb{\Xb}$ and $\Alp$ are $C^{*}$-algebras, so that this equality of norms is a simple
consequence of (i) since $C^{*}$-homomorphisms that preserve invertibility do also
preserve norms. In the general case such elegant arguments are not available anymore.}
the supremum (in fact maximum)
norm of \eqref{eq:opspec}. We refer to the norm of \eqref{eq:coset} as the essential norm of $A$.
\item[(v)] By a combination of (iii) and (iv), one derives
\begin{equation}\label{eq:resolvents}
\|[(A-\lambda I)+\pc{\Xb}]^{-1}\| = \max_{A_h\in\opsp(A)} \|(A_h-\lambda I)^{-1}\|
\end{equation}
for all $\lambda\in\Cb$, which is an equality between corresponding resolvent norms in \eqref{eq:coset}
and \eqref{eq:opspec}. This in turn proves the following pseudospectral version of \eqref{eq:specess}
\begin{equation}\label{eq:specesseps}
\speess (A) = \bigcup_{A_h\in\opsp(A)} \spe (A_h),\qquad \eps>0.
\end{equation}
Here $\speess (A)$ is the set of all $\lambda\in\Cb$ for which the left-hand side of \eqref{eq:resolvents}
is greater than $\frac 1\eps$, and $\spe (A_h)$ is the usual pseudospectrum of $A_h$ that we will discuss
in a minute. We will see that $\speess (A)$ is the pseudospectrum, in the same sense, of the coset \eqref{eq:coset};
it will henceforth be referred to as the essential pseudospectrum of $A$.
\end{enumerate}
Here is an important superset of the spectrum: For an operator $A\in\lb{\Xb}$ or, more generally,
an element $a$ in a Banach algebra $B$ with unit $e$, it is sometimes a more sensible question
to ask whether the inverse of $a-\lambda e$ is large in norm, possibly non-existent, rather than
just to ask for the latter. So one defines the {\sl $\eps$-pseudospectrum} of $a$ by
\[
\spe (a) := \left\{\lambda\in \Cb: \|(a-\lambda e)^{-1}\|>\frac 1\eps\right\}, \qquad\eps>0,
\]
where we say $\|b^{-1}\|:=\infty$ if $b$ is non-invertible, so that $\spc (a)\subset \spe (a)$.
This defines both $\spe (A)$ as the pseudospectrum of an operator $A$ in $B=\lb{\Xb}$ and
$\speess (A)$ as the pseudospectrum of a coset \eqref{eq:coset} in the quotient algebra
$B=\Alp/\pc{\Xb}$. For $A\in\lb{\Xb}$ there is the remarkable coincidence (see e.g.
\cite[Theorem 7.4]{BandToep})
\begin{equation}\label{eq:ps_pert}
\spe (A) = \bigcup_{\|T\|<\eps} \spc (A+T)
\end{equation}
for all $\eps>0$, showing that $\spe (A)$ exactly measures the sensitivity of $\spc (A)$ with respect to
additive perturbations of $A$ by operators $T\in\lb{\Xb}$ of norm less than $\eps$. For normal operators
$A$ on Hilbert space, $\spe (A)$ is exactly the $\eps$-neighbourhood of $\spc (A)$. In general it can be much
larger. Pseudospectra are interesting objects by themselves since they carry more information than spectra
(e.g. about transient instead of just asymptotic behaviour of dynamical systems). Also, they have
better convergence and approximation properties than spectra ($\spe (A)$ depends continuously on
$A$ -- unlike $\spc (A)$). Still, the $\eps$-pseudospectra approximate the spectrum as $\eps\to 0$.
\medskip

On the other hand, there is the ($\Pc$-)essential spectrum $\spess (A)$. This set is robust under
($\Pc\textrm{-}$)compact perturbations, enabling its study by means of limit operators via \eqref{eq:specess}.
\medskip

The essential pseudospectrum, $\speess (A)$, nicely blends these properties of essential and pseudospectra:
We have already mentioned that it inherits an $\eps$-version, \eqref{eq:specesseps}, of \eqref{eq:specess}.
We will also show that there is an essential version of \eqref{eq:ps_pert}, that is
\begin{equation}\label{eq:essps_pert}
\speess (A) = \bigcup_{\|T\|<\eps} \spess (A+T)
\end{equation}
for all $\eps>0$, where the perturbations $T$ come from $\Alp$. So in this new setting, the different
properties \eqref{eq:specess} and \eqref{eq:ps_pert} both generalize and meet in one place.

Besides these aesthetical aspects, our argument
for the study of $\speess (A)$ is as follows:
When $\spess (A)$ is of interest, the problem with formula \eqref{eq:specess} is the computation
of all limit operators $A_h$ and then of their spectra $\spc (A_h)$. It appears more feasible, from a
numerical perspective, to compute the pseudospectra $\spe (A_h)$ for small values of $\eps$, then
derive $\speess (A)$ by \eqref{eq:specesseps} and finally use that the closure of $\speess (A)$
tends to $\spess (A)$ in Hausdorff metric as $\eps\to 0$.

\subparagraph{Previous work}
The story of limit operators probably began in the late 1920's in Favard's paper \cite{Favard}
for studying ODEs with almost-periodic coefficients. It continued in the work of Muhamadiev
\cite{Muh1, Muh2, Muh3, Muh4} and was later followed by Lange and Rabinovich \cite{LangeR},
who were the first to consider Fredholmness for the generic class of band-dominated operators.
In the last 20 years, major work was done by Rabinovich, Roch, Roe and Silbermann
\cite{RochFredTh,LimOps,RochFredInd} with recent contributions by some of the authors
and Chandler-Wilde \cite{Marko,LiChW,SeDis,SeFre,BigQuest}.
A detailed review of this history is, for example, in the introduction of \cite{LiChW}.
A comprehensive presentation of these results, further achievements and applications
e.g. to convolution and pseudo-differential operators, as well as the required tools,
can be found in the 2004 book \cite{LimOps} of Rabinovich, Roch and Silbermann.
This literature shows that the list of parallels between the items \eqref{eq:coset}
and \eqref{eq:opspec} is actually longer than our list (i)--(v).
For example, in \cite{RochFredInd} it is shown for the case $\Xb=l^2(\Zb^1,\Cb)$
that the Fredholm index of $A$ can be recovered from two fairly arbitrary elements
of \eqref{eq:opspec}.

Apart from the theory of limit operators, there is of course a vast amount of literature
on spectral theory. Particularly related is the work of Trefethen, Embree and others
(see \cite{TrefEmbr} and references therein) on pseudospectra. We probably have to mention \cite{AmmarJeribi13,AmmarJeribi14}, where essential pseudospectra have been defined.

\subparagraph{Summary of contents}
In Section \ref{sec:defs} we summarize the main definitions and previously known results
including (i), (ii) and (iii) from above. Section \ref{sec:essnorm} is devoted to the proof of (iv).
It is then straightforward to conclude (v).
Section \ref{sec:essps} introduces and gives basic results about essential pseudospectra.
Section \ref{sec:esslowernorm} turns the attention to the so-called lower norm $\nu(B)$ of an operator $B$,
which is the infimum of $\|Bx\|$ over all $x$ with $\|x\|=1$. While the norm $\|B^{-1}\|$ of the inverse
(if existent), as in the right-hand side of \eqref{eq:resolvents}, can be expressed as $1/\nu(B)$, it is the
subject of Section~\ref{sec:esslowernorm} to characterize (or at least bound) also the essential norm of the resolvent
on the left-hand side of \eqref{eq:resolvents} by means of lower norms, without explicit reference to limit operators -- thereby giving different
approaches to the computation of the essential pseudospectrum (or to upper and lower bounds on
it).
In Section \ref{sec:FSM} we discuss an application of our results in the context of
approximation methods. For reasons of numerical viability, an operator $A\in\Alp$ is usually
approximated by finite-dimensional operators $A_n$, hoping that their inverses $A_n^{-1}$
will exist and approximate $A^{-1}$, provided the latter exists. The key question here is about
the stability of the sequence $(A_n)_{n\in\Nb}$ -- and that can be translated into the language
of $\Pc$-Fredholmness of an associated operator. Here our previous results yield new quantitative insights, in particular (partial) limits of the norms $\|A_n^{-1}\|$ and the condition numbers $\kappa(A_n)$.

\section{Definitions and known results} \label{sec:defs}
In this section we give the relevant definitions and state corresponding theorems
about what has been known, including points (i), (ii) and (iii) from the introduction.

\subparagraph{Banach spaces and projections}
Throughout this paper $\Yb$ denotes a complex Banach space.
On $\Yb$ we look at a sequence $\Pc=(P_n)_{n\in\Nb}$ of projections $P_n$ with the properties
\begin{itemize}
\itemsep-1mm
\item[$(\Pc1)$] $P_n=P_nP_{n+1}=P_{n+1}P_n$ and $\|P_n\|=\|Q_n\|=1$ for all $n\in\Nb$, where $Q_n:=I-P_n$;
\item[$(\Pc2)$] $C_\Pc:=\sup_{U\subset\Nb} \|\sum_{n\in U} (P_{n+1}-P_n)\|<\infty$ with the supremum taken over all finite sets $U\subset\Nb$.
\end{itemize}
Then $\Pc$ is a so-called {\sl uniform approximate projection} in the sense of \cite{LimOps,SeSi3,SeDis}.
If additionally
\begin{itemize}
\item[$(\Pc3)$] $\sup_{n\in\Nb}\|P_nx\|\ge\|x\|$ for all $x\in\Yb$
\end{itemize}
then $\Pc$ is said to be a {\sl uniform approximate identity}.
Note that the dual sequence $\Pc^*=(P_n^*)$ then also has the corresponding properties $(\Pc1)$ and $(\Pc2)$ on the dual space $\Yb^*$ but not necessarily $(\Pc3)$.

However, in large parts we consider the more particular situation of a (generalized)
sequence space $\Xb=l^p(\Zb^N,X)$ with parameters $p\in\{0\}\cup[1,\infty]$,
$N\in\Nb$ and a complex Banach space $X$. These (generalized) sequences are of the form
$x=(x_i)_{i\in\Zb^N}$ with all $x_i\in X$. The spaces are equipped with the usual
$p$-norm. In our notation, $l^0(\Zb^N,X)$ stands for the closure in
$l^\infty(\Zb^N,X)$ of the set of all sequences $(x_i)_{i\in\Zb^N}$ with finite support.
In the context of these sequence spaces $\Xb$, $\Pc=(P_n)$ shall always be the
sequence of the canonical projections\footnote{Here and in what follows we write
$\chi_M:\Zb^N\to\{0,1\}$ for the characteristic function of $M\subset\Zb^N$,
that is $\chi_M(k)=1$ if $k\in M$ and $=0$ otherwise.}
$P_n:=\chi_{\{-n,\ldots,n\}^N}I$
which obviously forms a uniform approximate identity on $\Xb$.

Notice that this variety of spaces $\Xb=l^p(\Zb^N,X)$ in particular covers the
spaces $L^p(\Rb^N)$ by the natural isometric identification of a function in
$L^p(\Rb^N)$ with the sequence of its restrictions to the hypercubes $i+[0,1]^N$
with $i\in\Zb^N$ (saying that $L^p(\Rb^N)\cong l^p(\Zb^N,L^p([0,1]^N))$).

\subparagraph{Operators and convergence}

The following definitions and results are taken from e.g. \cite{LimOps, Marko, SeFre}.
Starting with a Banach space $\Yb$ and a uniform approximate projection $\Pc=(P_n)$,
one says that a bounded linear operator $K$ on $\Yb$ is {\sl $\Pc$-compact} if
\begin{equation}\label{eq:pcomp}
\|(I-P_n)K\|+\|K(I-P_n)\|\to 0\quad\text{as}\quad n\to\infty.
\end{equation}
The set\footnote{This set is the closure (in $\lb{\Yb}$) of the
set $\pcn{\Yb}$ of all $K\in\lb{\Yb}$ for which $\|(I-P_n)K\|+\|K(I-P_n)\|= 0$ for
all sufficiently large $n$. $\pcn{\Yb}$ corresponds to matrices with finite support
if $\Yb=\Xb=l^p(\Zb^N,X)$.} of all $\Pc$-compact operators is denoted by $\pc{\Yb}$.
Unlike the set $\lc{\Yb}$ of all compact operators on $\Yb$, $\pc{\Yb}$ is in general not
an ideal in $\lb{\Yb}$. So we introduce the following subset of $\lb{\Yb}$:
\[
\pb{\Yb}:=\{A\in\lb{\Yb}:AK,KA\in\pc{\Yb}\text{ for all }K\in\pc{\Yb}\}.
\]
Now $\pb{\Yb}$ forms a closed subalgebra of $\lb{\Yb}$ containing $\pc{\Yb}$
as a closed two-sided ideal.

Generalizing usual Fredholmness and the essential spectrum, one now studies invertibility
modulo $\pc{\Yb}$: An operator $A\in\lb{\Yb}$ is said to be {\sl invertible at infinity}
if there is a so-called {\sl $\Pc$-regularizer} $B\in\lb{\Yb}$ such that $AB-I$ and $BA-I$
are $\Pc$-compact. Similarly, $A\in\pb{\Yb}$ is called {\sl $\Pc$-Fredholm} if the coset
$A+\pc{\Yb}$ is invertible in the quotient algebra $\pb{\Yb}/\pc{\Yb}$.
For $A\in\pb{\Yb}$, the {\sl $\Pc$-essential spectrum} $\spess (A)$ is then the set
of all $\lambda\in\Cb$ for which $A-\lambda I$ is not $\Pc$-Fredholm.

\begin{thm}\label{TPFredh}(\cite[Theorem 1.16]{SeSi3})\\
An operator $A\in\pb{\Yb}$ is $\Pc$-Fredholm if and only if it is invertible at
infinity. In this case every $\Pc$-regularizer of $A$ belongs to $\pb{\Yb}$.
Particularly, $\pb{\Yb}$ is inverse closed in $\lb{\Yb}$.
\end{thm}

Finally, if $\Pc$ is a uniform approximate identity,
say that a sequence $(A_n)\subset\lb{\Yb}$ converges {\sl $\Pc$-strongly} to an
operator $A\in\lb{\Yb}$ if
\begin{equation}\label{eq:pto}
\|K(A_n-A)\|+\|(A_n-A)K\|\to0\quad\text{as}\quad n\to\infty
\end{equation}
for every $K\in\pc{\Yb}$.
We shortly write $A_n\pto A$ or $A=\plim A_n$ in that case.
Note that \eqref{eq:pcomp} and \eqref{eq:pto} immediately imply\footnote{It becomes
clear that $\pc{\Yb}$ and $\pto$ are actually tailor-made by \eqref{eq:pcomp}
and \eqref{eq:pto} for this purpose.} $P_n\pto I$.
By \cite[Theorem 1.65]{Marko}, $A_n\pto A$ is equivalent to
the sequence $(A_n)$ being bounded and
\[
\|P_m(A_n-A)\|+\|(A_n-A)P_m\|\to0\quad\text{as}\quad n\to\infty
\]
for every $m\in\Nb$.  Also note \cite[Prop. 1.1.17]{LimOps} that for the $\Pc$-limit $A$ one has
\begin{equation}\label{eq:liminf}
A\in\pb{\Yb}\qquad\textrm{and}\qquad\|A\|\leq \liminf\|A_n\|
\end{equation}
if all $A_n$ are in $\pb{\Yb}$.

$\Pc$-compactness determines the notions of $\Pc$-convergence and $\Pc$-Fredholmness just like
compactness does with strong convergence and the usual Fredholmness \cite[Section 1.1]{LimOps}.

\begin{rem}\label{RFinDim}
For $\Yb=\Xb=l^p(\Zb^N,X)$, one has $\pc{\Xb}\supset\lc{\Xb}$ if $p\in(1,\infty)$, whereas
$\pc{\Xb}\subset\lc{\Xb}$ if $\dim X<\infty$. So for $p\in(1,\infty)$ and $\dim X<\infty$,
the $\Pc$-notions coincide with the classical ones: $\pc{\Xb}=\lc{\Xb}$, $\pb{\Xb}=\lb{\Xb}$,
an operator is $\Pc$-Fredholm if and only if it is Fredholm, and
a sequence $(A_n)$ converges $\Pc$-strongly to $A$ if and only if $A_n\to A$ and
$A_n^*\to A^*$ strongly.

The reason for the definition of the $\Pc$-notions is to extend the well-known
concepts, tools and connections between them in a way that they
still apply to relevant operators and sequences in the cases $p\in\{1,\infty\}$ and/or
$\dim X=\infty$. For example, although $P_n\not\to I$ if $p=\infty$ and $P_n^*\not\to I^*$
if $p=1$, one still has $P_n\pto I$ in all cases. Also, each $P_n$ is $\Pc$-compact,
although not compact, in case $\dim X=\infty$.
\end{rem}

Anyway, on $\Xb$, the (classical) Fredholm property nicely fits into the generalized
$\Pc$-setting:
\begin{prop}\label{PFredh}
\cite[Corollary 12]{SeFre}\\
Let $A\in\pb{\Xb}$ be Fredholm. Then $A$ is $\Pc$-Fredholm and has a
generalized inverse $B\in\pb{\Xb}$, i.e. $A=ABA$ and $B=BAB$. Moreover, $A$ is
Fredholm of index zero if and only if there exists an invertible operator
$C\in\pb{\Xb}$ and an operator $K\in\pc{\Xb}$ of finite rank such that $A=C+K$.
\end{prop}

\subparagraph{Equivalent approximate projections}

If we fix an approximate projection $\Pc$ and an operator $A\in\pb{\Yb}$, we can always find an equivalent approximate projection that is tailored for $A$. This provides noticeable simplifications in many arguments.

\begin{prop} \label{PFn}
(extension of \cite[Theorem 1.15]{SeSi3})\\
Let $\Pc$ be a uniform approximate projection on $\Yb$ and $A\in\pb{\Yb}$. \\
Then there exists a sequence $\hat{\Pc} = (F_n)$ of operators that satisfies $(\Pc 1)$ and $(\Pc 2)$ with $C_{\hat{\Pc}} \leq C_{\Pc}$, and for every $n\in\Nb$ there exists $m\in\Nb$ such that $F_nP_m=P_mF_n=F_n$ as well as $P_nF_m=F_mP_n=P_n$, and $\|[A,F_n]\|=\|AF_n-F_nA\|\to 0$ as $n\to\infty$.\\
If $\Pc$ is a uniform approximate identity, then $\lim_n\|F_nx\|=\|x\|$ for every $x\in\Yb$.
\end{prop}

\begin{proof}
The existence of $(F_n)$ with $F_nP_m=P_mF_n=F_n$ and $P_nF_m=F_mP_n=P_n$ as announced,
and $\|[A,F_n]\|=\|AF_n-F_nA\|\to 0$ as $n\to\infty$ was proved in
\cite[Theorem 1.15]{SeSi3}. Actually, for each $n\in\Nb$, these $F_n$ are of the form
(see \cite[Equation (1.4)]{SeSi3} and the proof there)
\[F_n=\frac1n\sum_{k=1}^nkP_{U_{n-k}^n}=\frac1n\left(\sum_{k=1}^{n-1} k(P_{r_{n-k+1}^n}-P_{r_{n-k}^n}) + nP_{r_1^n}\right)=\frac1n\sum_{k=1}^nP_{r_k^n}\]
with certain integers $1<r_1^n<r_2^n<\ldots<r_n^n$. Thus,
\[1=\|P_1\|=\|P_1F_n\|\leq\|F_n\|\leq \frac1n\sum_{k=1}^n\|P_{r_k^n}\|=\frac nn =1.\]
Similarly, for every $n \in \Nb$ there exists $m \in \Nb$ such that
\[1=\|I-P_m\|=\|(I-P_m)(I-F_n)\|\leq\|I-F_n\|\leq \frac1n\sum_{k=1}^n\|I-P_{r_k^n}\|=\frac{n}{n}=1.\]
For $F_n = F_nF_{n+1} = F_{n+1}F_n$ and $(\Pc 2)$ see again \cite[Theorem 1.15]{SeSi3}.
Finally, since $\|F_nx\|=\|F_nP_mx\|\leq\|P_mx\|$ and $\|P_nx\|=\|P_nF_mx\|\leq \|F_mx\|$ for
$m\gg n$ we have
\[\sup_n\|F_nx\|=\lim_n\|F_nx\|=\lim_n\|P_nx\|=\sup_n\|P_nx\|\]
for each $x\in\Yb$. Hence if additionally $(\Pc3)$ is fulfilled then $\lim_n\|F_nx\|=\|x\|$
for every $x\in\Yb$.
\end{proof}

\subparagraph{Band and band-dominated operators}
Every sequence $a=(a_n)\in l^\infty(\Zb^N,\lb{X})$ gives rise to an operator
$aI\in\lb{\Xb}$, a so-called multiplication operator, via the rule
$(ax)_i=a_ix_i$, $i\in\Zb^N.$
For every $\alpha\in\Zb^N$, we define the shift operator
$V_\alpha:\Xb\to\Xb$, $(x_i)\mapsto(x_{i-\alpha})$.

A band operator is a finite sum of the form $\sum a_\alpha V_\alpha$,
where $a_\alpha I$ are multiplication operators.
In terms of the generalized matrix-vector multiplication
\[
(a_{ij})_{i,j\in\Zb^N}\ (x_j)_{j\in\Zb^N}=(y_i)_{i\in\Zb^N}
\quad\textrm{with}\quad
y_i\ =\ \sum_{j\in\Zb^N}a_{ij}x_j,\ i\in\Zb^N,\quad\textrm{where\ } a_{ij}\in\lb{X},
\]
band operators $A$ act on $\Xb=l^p(\Zb^N,X)$ via multiplication by band matrices $(a_{ij})$,
that means $a_{ij}=0$ if $|i-j|$ exceeds the so-called band-width of $A$.
Typical examples are discretizations of differential operators on $\Rb^N$.

In many physical models, however, interaction $a_{ij}$ between data at locations
$i$ and $j$ decreases in a certain way as $|i-j|\to\infty$ rather than suddenly stop
at a prescribed distance of $i$ and $j$. An operator is called {\sl band-dominated}
if it is contained in the $\lb{\Xb}$-closure, denoted by $\Alp$, of the set $\Aln$
of all band operators .
In contrast to $\Aln$ (which is an algebra but not closed in $\lb{\Xb}$),
the set $\Alp$ is a Banach algebra, for which the inclusions
\[
\pc{\Xb}\subset\Alp\subset\pb{\Xb}\subset\lb{\Xb}
\]
hold. In particular, $\pc{\Xb}$ is a two-sided closed ideal in $\Alp$.

\begin{thm}\label{TBDOInvCl}
\cite[Propositions 2.1.7 et seq.]{LimOps}\\
Let $A\in\Alp$ be $\Pc$-Fredholm. Then every $\Pc$-regularizer of $A$ is band-dominated
as well. In particular, the quotient algebra $\Alp/\pc{\Xb}$ is inverse closed in $\pb{\Xb}/\pc{\Xb}$,
and $\Alp$ is inverse closed in both $\pb{\Xb}$ and $\lb{\Xb}$.
\end{thm}

So for $A\in\Alp$, the following are equivalent:
\begin{itemize}
\itemsep-1mm
\item
$A$ is invertible at infinity (i.e. it has a $\Pc$-regularizer in $\lb{\Xb}$),
\item
$A$ is $\Pc$-Fredholm (it has a $\Pc$-regularizer in $\pb{\Xb}$),
\item
the coset \eqref{eq:coset} is invertible ($A$ has a $\Pc$-regularizer in $\Alp$).
\end{itemize}

The first studies of particular subclasses of band-dominated operators and their Fredholm properties were for
the case of constant matrix diagonals, that is when the matrix entries $a_{ij}$ only depend on the difference $i-j$,
so that $A$ is a convolution
operator (a.k.a. Laurent or bi-infinite Toeplitz matrix, the stationary case) \cite{GoKre,Sim,GoFe,HaRoSi,Analysis}. Subsequently, the focus went to more general classes, such as convergent, periodic and almost periodic matrix diagonals, until at the current point arbitrary matrix diagonals can be studied -- as long as they are bounded. This possibility is due to the notion of limit operators that enables evaluation of the asymptotic behavior of an operator $A$ even for merely bounded diagonals in the matrix $(a_{ij})$.

\subparagraph{Limit operators}

Say that a sequence $h=(h_n)\subset \Zb^N$ tends to infinity if $|h_n|\to\infty$ as
$n\to\infty$.
If $h=(h_n)\subset \Zb^N$ tends to infinity and $A\in\pb{\Xb}$ then
\[A_h:=\plimn V_{-h_n}AV_{h_n},\]
if it exists, is called the {\sl limit operator} of $A$ w.r.t.~the sequence $h$.
The set \eqref{eq:opspec}
of all limit operators of $A$ is its {\sl operator spectrum}, $\opsp(A)$.

\begin{prop} \label{prop:limops}\cite[Proposition 1.2.2]{LimOps}
Let $A,B\in\pb{\Xb}$ and $h=(h_n)\subset\Zb^N$ tend to infinity such that $A_h$ and $B_h$ exist. Then:
\begin{itemize}
\itemsep-1mm
\item also $(A+B)_h$ and $(AB)_h$ exist, where $(A+B)_h=A_h+B_h$ and $(AB)_h=A_hB_h$;
\item if $p<\infty$, also $(A^*)_h$ exists and equals $(A_h)^*$;
\item the inequality $\|A_h\|\le\|A\|$ holds.
\end{itemize}
\end{prop}

\begin{thm}\label{TLimOps}
\cite[Theorem 16]{SeFre}\\
Let $A\in\pb{\Xb}$ be $\Pc$-Fredholm.
Then all limit operators of $A$ are invertible and their inverses are uniformly
bounded. Moreover, the operator spectrum of every $\Pc$-regularizer $B$ of $A$ equals
\begin{equation}\label{ERegOpSpec}
\opsp(B) = \{A_h^{-1}:A_h\in\opsp(A)\}.
\end{equation}
\end{thm}

We say that $A\in\pb{\Xb}$ has a rich operator spectrum (or we simply call $A$
a {\sl rich operator}) if every sequence $h\subset\Zb^N$ tending to infinity has
a subsequence $g\subset h$ such that the limit operator  $A_g$ of $A$ w.r.t.~$g$
exists. The set of all rich operators $A\in\pb{\Xb}$ is denoted by $\pbr{\Xb}$
and the set of all rich band-dominated operators by $\Alpr$. Recall from
\cite[Corollary 2.1.17]{LimOps}  that $\Alpr=\Alp$ whenever $\dim X<\infty$.
For rich operators we know

\begin{thm}\label{TRich}
\cite[Corollary 17]{SeFre}
\begin{itemize}
\item The set $\pbr{\Xb}$ forms a closed subalgebra of $\pb{\Xb}$ and contains
      $\pc{\Xb}$ as a closed two-sided ideal.
\item Every $\Pc$-regularizer of a rich $\Pc$-Fredholm operator is rich. Thus,
      $\pbr{\Xb}/\pc{\Xb}$ is inverse closed in $\pb{\Xb}/\pc{\Xb}$ and
      $\pbr{\Xb}$ is inverse closed in both $\pb{\Xb}$ and $\lb{\Xb}$.
\end{itemize}
\end{thm}

In the case of rich band-dominated operators, the picture is most complete:

\begin{thm}\label{TBdORichHS}
(\cite{RochFredTh}, \cite[Theorem 6.28]{LiChW} and \cite{BigQuest})\\
For an operator $A\in\Alpr$, the following are equivalent:
\begin{itemize}
\itemsep-1mm
\item $A$ is $\Pc$-Fredholm,
\item all limit operators of $A$ are invertible and their inverses are uniformly bounded,
\item all limit operators of $A$ are invertible.
\end{itemize}
\end{thm}

This is result (i) from the introduction. Equality \eqref{eq:specess} from point (ii)
follows by replacing $A$ by $A-\lambda I$ in Theorem \ref{TBdORichHS}, noting
that $\opsp(A-\lambda I)=\opsp(A)-\lambda I$. Furthermore, (iii) is an immediate
consequence of Theorem \ref{TLimOps}.

\subparagraph{The lower norm}
For an operator $A$ between two Banach spaces, we call
\[
\nu(A)\ :=\ \inf\{\|Ax\|:\|x\|=1\}
\]
the {\sl lower norm} of $A$.
For operators on Hilbert space, $\nu(A)$ is the smallest singular value of $A$.
We call $A$ {\sl bounded below} if $\nu(A)>0$. The following properties are well known.
A proof can be found e.g. in \cite[Lemmas 2.32, 2.33 and 2.35]{Marko}.
\newpage 

\begin{lem}\label{lem:bddbelow}
Let $\Yb_1$ and $\Yb_2$ be Banach spaces and $A:\Yb_1\to\Yb_2$ be a bounded linear operator between them. Then the following hold:
\begin{itemize}
\itemsep-1mm
\item $A$ is bounded below iff $A$ is injective and its range is closed in $\Yb_2$;
\item $A$ is invertible from the left iff $A$ is injective and its range is complementable in $\Yb_2$, in which case $\nu(A)\geq\|A^{l}\|^{-1}$ for every left inverse $A^l$;
\item $A$ is invertible iff $A$ and $A^*$ are bounded below, in which case $\nu(A)=\|A^{-1}\|^{-1}$.
\end{itemize}
More generally,\footnote{From now on we set $\infty^{-1}:=0$, so that
$\|b^{-1}\|^{-1}=0$ if $b$ is not invertible.} $\|A^{-1}\|^{-1}=\min\{\nu(A),\nu(A^*)\}$, where
$\nu(A)=\nu(A^*)$ if both are positive.
\end{lem}

\section{\texorpdfstring{$\Pc$}{P}-essential norm of band-dominated operators} \label{sec:essnorm}
In this section we prove the first new result, point (iv) from the introduction,
about the norm of the coset \eqref{eq:coset}. As an immediate consequence
of (iii) and (iv) we get the first half of (v). Recall that we abbreviate $I-P_n$ by $Q_n$.

\begin{prop} \label{PEssNorm}
Let $\Pc$ be a uniform approximate projection on $\Yb$ and $A \in \pb{\Yb}$. Then
\[\|A+\pc{\Yb}\| = \|A^*+\lc{\Yb^*,\Pc^*}\|
= \lim_{m\to\infty} \|AQ_m\| = \lim_{m\to\infty} \|Q_m A\|.\]
\end{prop}
\begin{proof}
Let $\eps > 0$ and choose $K \in \pc{\Yb}$ such that
$\|A+K\| \leq \|A+\pc{\Yb}\| + \eps$ and $m_0\in\Nb$ such that
$\|KQ_m\| \leq \eps$ for all $m \geq m_0$. It follows
\[\|AQ_m\| = \|A-AP_m\| \geq\|A+\pc{\Yb}\| \geq \|A+K\| - \eps
\geq \|(A+K)Q_m\| - \eps \geq \|AQ_m\| - 2\eps\]
for all $m \geq m_0$ and therefore $\|A+\pc{\Yb}\| = \lim_{m\to\infty} \|AQ_m\|$
since $\eps$ was arbitrary. The equality
$\|A+\pc{\Yb}\| = \lim_{m\to\infty} \|Q_mA\|$ is similar. Finally,
$\|A^*Q_m^*\|=\|Q_mA\|$ finishes the proof.
\end{proof}

Now we switch to sequence spaces $\Xb=l^p(\Zb^N,X)$ and band-dominated operators.
Our first main theorem is
\begin{thm}\label{TEssNorm}
Let $A\in\Alpr$. Then
\begin{equation}\label{EEssNorm}
\|A+\pc{\Xb}\|=\max_{A_g\in\opsp(A)}\|A_g\|.
\end{equation}
\end{thm}

Note that if $\Xb$ is a Hilbert space, $C^*$-algebra techniques can be used to deduce Theorem \ref{TEssNorm} directly from Theorem \ref{TBdORichHS} (cf. \cite[Thm 2.2.7]{LimOps}). In the general case we require the following auxiliary notion:

\begin{defn}
The {\sl support} of a sequence $x=(x_n)\in \Xb$ is the set
$\supp x :=\{n\in\Zb^N:x_n\neq 0\}$. The diameter of a subset $M\subset \Zb^N$
is defined as $\diam M := \sup\{|n_i-m_i|:n,m\in M; i=1,\ldots, N\}$. Moreover,
we write $|M|$ for the number of elements of any set $M$. For $D\in\Nb$
we now define
\[\vertiii{A}_D:=\sup\left\{\frac{\|Ax\|}{\|x\|}: x\in\Xb\setminus\{0\},\,
	\diam\supp x\leq D\right\}.\]
\end{defn}

The first step is to show that the operator norm of $A\in\Aln$ can be localized, up to any
desired accuracy, in terms of $\vertiii{A}_D$ (i.e. by looking at sequences $x\in\Xb$
with support of a certain diameter) in the following sense:

\begin{prop}\label{PLoc}
Let $A\in\Aln$ and $\delta>0$. Then there is a $D\in\Nb$ such that
\[(1-\delta)\|A\chi_FI\|\leq\vertiii{A\chi_FI}_D\leq\|A\chi_FI\|
	\quad\text{for all}\quad F\subset\Zb^N.\]
\end{prop}

There is a very similar statement in \cite[Prop. 6]{BigQuest} for the lower norm $\nu(A)$
that we will address in Section \ref{sec:esslowernorm}.
Also the proof is very similar. In \cite{BigQuest} there are two different proofs given,
here we restrict ourselves to showing one of the two proofs that we know of (the one
that uses and generalizes a technique from \cite{CW.Heng.ML:UpperBounds}):

\begin{proof}
Clearly $\vertiii{\cdot}_D\leq\|\cdot\|$. So let $A\in\Aln$ and
let $w\in\Nb$ be its band-width, i.e. $\chi_U A\chi_V I=0$
for all $U,V\subset\Zb^N$ with $\dist(U,V):=\inf\{|u-v|_\infty:u\in U, v\in V\}>w$.

For arbitrary $n\in\Nb$ and $k\in\Zb^N$, put $C_n:=\{-n,...,n\}^N$, $C_{n,k}:=k+C_n$, 
$D_n:=C_{n+w}\setminus C_{n-w}$, $D_{n,k}:=k+D_n$, $c_n:=|C_n|=|C_{n,k}|=(2n+1)^N$ 
and $d_n:=|D_n|=|D_{n,k}|=c_{n+w}-c_{n-w}\sim n^{N-1}$.
Abbreviate $\chi_{C_{n,k}}I=:P_{n,k}$ and $\chi_{D_{n,k}}I=:\Delta_{n,k}$.

We start with the case $p\in [1,\infty)$. Given such $p$ and our arbitrary $\delta>0$, 
we choose $n\in\Nb$ large enough that $\frac{d_n}{c_n}<(\frac \delta 4)^p$.
Then $D:=2n+1$ will turn out to satisfy what we claim.

Now fix an arbitrary $F\subset\Zb^N$ and note that also $B:=A\chi_FI$ is a band operator of the same band-width $w$. W.l.o.g.~we may assume that $B \neq 0$.
We note the following facts:
\begin{itemize}
\item[(a)] For all finite sets $S\subset\Zb^N$ and all $x\in\Xb$, it holds
$
\sum_{k\in\Zb^N}\|\chi_{k+S}\,x\|^p = |S|\cdot \|x\|^p.
$
\item[(b)] For the commutator $[P_{n,k},B]:=P_{n,k}B-BP_{n,k}$, one has $[P_{n,k},B]=[P_{n,k},B]\Delta_{n,k}$, so that for all $x\in\Xb$,
$
\|[P_{n,k},B]x\| = \|[P_{n,k},B]\Delta_{n,k}x\| \le \|[P_{n,k},B]\| \|\Delta_{n,k}x\|  \le 2\|B\| \|\Delta_{n,k}x\|
$
and hence
\[
\sum_{k\in\Zb^N} \|[P_{n,k},B]x\|^p\ \le\ \sum_{k\in\Zb^N} 2^p\|B\|^p \|\Delta_{n,k}x\|^p\ \stackrel{(a)}{=}\ 2^p\|B\|^p d_n \|x\|^p.
\]
\end{itemize}

Fixing $x\in\Xb$ such that $(1-\frac \delta 2)\|B\| \|x\| < \|Bx\|$, we conclude as follows, where (M) refers to Minkowski's inequality in $l^p(\Zb^N,\Cb)$:
\begin{align*}
&\left(1-\frac \delta 2\right) \|B\| c_n^{1/p} \|x\|
\ <\  c_n^{1/p} \|Bx\| \ \stackrel{(a)}{=}\ \left(\sum_{k\in\Zb^N} \|P_{n,k}Bx\|^p\right)^{1/p}\\
&\le  \left(\sum_{k\in\Zb^N} \left(\|BP_{n,k}x\|+\|[P_{n,k},B]x\|\right)^p\right)^{1/p} \stackrel{({\rm M})}{\le}\  \left(\sum_{k\in\Zb^N} \|BP_{n,k}x\|^p\right)^{1/p} + \left(\sum_{k\in\Zb^N} \|[P_{n,k},B]x\|^p\right)^{1/p}\\
&\stackrel{(b)}{\le}\ \left(\sum_{k\in\Zb^N} \|BP_{n,k}x\|^p\right)^{1/p}
+\ 2\|B\| d_n^{1/p} \|x\| 
\end{align*}
Subtract $2 (d_n/c_n)^{1/p} \|B\|c_n^{1/p}\|x\|=2\|B\| d_n^{1/p} \|x\|$ from the resulting inequality, 
to get
\[
\left(1-\frac \delta 2 - 2 \left(\frac{d_n}{c_n}\right)^{1/p}\right) \|B\| c_n^{1/p} \|x\|\ <\
\left(\sum_{k\in\Zb^N} \|BP_{n,k}x\|^p\right)^{1/p}.
\]
Taking $p$-th powers, using $2 \left(\frac{d_n}{c_n}\right)^{1/p}< \frac \delta 2$ and $\sum_{k\in\Zb^N} \|P_{n,k}x\|^p=c_n\|x\|^p$, by (a), we get
\[
(1-\delta)^p\|B\|^p \sum_{k\in\Zb^N} \|P_{n,k}x\|^p\
<\ \left(1-\frac \delta 2 - 2 \left(\frac{d_n}{c_n}\right)^{1/p}\right)^p \|B\|^p c_n\|x\|^p\
<\ \sum_{k\in\Zb^N} \|BP_{n,k}x\|^p.
\]
The last inequality shows that there must be some $k\in\Zb^N$ for which
$P_{n,k}x\neq0$ and
\[
(1-\delta)^p\|B\|^p \|P_{n,k}x\|^p < \|BP_{n,k}x\|^p,\textrm{\quad i.e.\quad }
(1-\delta)\|B\| \|P_{n,k}x\| < \|BP_{n,k}x\| \le \vertiii{B}_D\|P_{n,k}x\|
\]
with $D=2n+1$ as fixed above. This finishes the proof for $p\in [1,\infty)$. 
Finally, let $p\in\{0,\infty\}$, put $D:=2w+1$, take any $F\subset\Zb^N$, $B:=A\chi_FI$, 
$\eps>0$ and $x\in \Xb$ with $\|x\|_\infty=1$ and $\|Bx\|_\infty\ge\|B\|-\eps/2$. Then there is a $k\in\Zb^N$ with $\|\chi_{\{k\}}Bx\|_\infty\ge\|Bx\|_\infty-\eps/2$, so that
\begin{align*}
\|B\|-\eps
&\ \leq\ \|Bx\|_\infty-\eps/2
\ \leq\ \|\chi_{\{k\}}Bx\|_\infty
\ =\ \|\chi_{\{k\}}BP_{w,k}x\|_\infty
\ \leq\  \|BP_{w,k}x\|_\infty\\
&\ \leq\ \vertiii{B}_D\|P_{w,k}x\|_\infty
\ \leq\ \vertiii{B}_D\|x\|_\infty
\ =\ \vertiii{B}_D\ \le\ \|B\|
\end{align*}
holds. So in case $p\in\{0,\infty\}$ even equality $\|B\|=\vertiii{B}_D$ follows, where $D=2w+1$.
\end{proof}

A closer look at this proof shows that the size of the support that is required to localize
the norm of $B$ to the desired accuracy only depends on the band-width $w$ of $B$,
so that the result carries over in a uniform way to all band operators with band-width not more than $w$. In short:
\begin{equation*} \label{eq:ONL} \tag{ONL}
\forall w\in\Nb,\, c\in(0,1)\ \exists D\in\Nb: \forall B \textrm{ with band-width}(B)\le w:\ \vertiii{B}_D\ge c\|B\|.
\end{equation*}
This localizability of the operator
norm is no longer a property of a particular operator but rather of the space $\Xb$.
There is recent work by X.~Chen, R.~Tessera, X.~Wang, G.~Yu and H.~Sako
(see \cite{ONL} and references therein) on metric spaces $M$ with a certain
measure such that $\Xb=l^2(M)$ has the operator norm localization property
\eqref{eq:ONL}. Sako proves in \cite{ONL} that in case of a discrete metric
space $M$ with $\sup_{m\in M}|\{n\in M:d(m,n)\le R\}|<\infty$ for all radii
$R>0$  (which clearly holds in our case, $M=\Zb^N$), property \eqref{eq:ONL}
is equivalent to the so-called Property~A that was introduced by G.~Yu and is
connected with amenability. We also want to mention the very recent paper
\cite{SpakWillett} by \v{S}pakula and Willett that generalizes the limit operator
results from $\Zb^N$ to certain discrete metric spaces. Based on the work of
Roe \cite{Roe}, combined with ideas of \cite{BigQuest}, they prove a version
of Theorem \ref{TBdORichHS} under the sole assumption that these metric spaces
have Yu's Property A.

For the current paper we are not interested in extending Proposition \ref{PLoc}
to band operators of a certain band-width but rather to the operator spectrum of an operator $A\in\Alp$:
\begin{cor}\label{CLoc}
Let $A\in\Alp$ and $\delta>0$. Then there is a $D\in\Nb$ such that
\[ \|B\chi_FI\|-\delta\leq\vertiii{B\chi_FI}_D\leq\|B\chi_FI\|
	\quad\text{for all}\quad F\subset\Zb^N\quad\text{and all}\quad B\in\{A\}\cup\opsp(A).\]
\end{cor}
\begin{proof}
Fix $\delta > 0$ and take a band operator $\tilde{A}$ such that $\|A-\tilde{A}\| < \delta/3$. Now choose $D$ by applying the previous proposition to $\tilde{A}$ with $\frac{\delta}{3\|\tilde{A}\|}$ instead of $\delta$. Then, for all $F\subset\Zb^N$,
\begin{align*}
\|A\chi_FI\| &\geq \vertiii{A\chi_FI}_D \geq \vertiii{\tilde{A}\chi_FI}_D - \vertiii{(A-\tilde{A})\chi_FI}_D > \left(1-\frac{\delta}{3\|\tilde{A}\|}\right)\|\tilde{A}\chi_FI\| - \frac{\delta}{3}\\
&\geq \|\tilde{A}\chi_FI\| - \frac{2\delta}{3} \geq \|A\chi_FI\| - \|(A-\tilde{A})\chi_FI\| - \frac{2\delta}{3} > \|A\chi_FI\| - \delta.
\end{align*}
Now let $A_g\in\opsp(A)$.
The estimate $\vertiii{A_g\chi_FI}_D\leq\|A_g\chi_FI\|$ is clear. Further, for every
$\eps>0$ there is an $m$ such that $\|A_g\chi_FI\| \leq \|A_g\chi_F P_m\| + \eps$.
For $\|A_g\chi_F P_m\|$ we have the estimate
\begin{align*}
\|A_g\chi_FP_m\|
&\leq \|V_{-g_n}AV_{g_n}\chi_FP_m\| + \|(A_g-V_{-g_n}AV_{g_n})P_m\|\\
&= \|A\chi_{F\cap\{-m,\ldots,m\}^N+g_n}I\| + \|(A_g-V_{-g_n}AV_{g_n})P_m\|\\
&\leq \vertiii{A\chi_{F\cap\{-m,\ldots,m\}^N+g_n}I}_D + \delta + \|(A_g-V_{-g_n}AV_{g_n})P_m\|\\
&= \vertiii{V_{-g_n}AV_{g_n}\chi_FP_m}_D  + \delta + \|(A_g-V_{-g_n}AV_{g_n})P_m\|.
\end{align*}
The last summand goes to zero as $n\to\infty$, whereas the 1st one converges to
$\vertiii{A_g\chi_FP_m}_D$.
By this we obtain
$\|A_g\chi_FI\| - \delta \leq \vertiii{A_g\chi_FP_m}_D + \eps \leq \vertiii{A_g\chi_F I}_D + \eps$
where $\eps>0$ is arbitrary. Thus the assertion follows.
\end{proof}

Now we are in a position to prove Theorem \ref{TEssNorm}.
\begin{proof}[Proof of Theorem \ref{TEssNorm}]
For every $K\in\pc{\Xb}$ and every $A_g\in\opsp(A)$,
$\|A+K\|\ge\|(A+K)_g\|=\|A_g\|$
holds. Taking the infimum on the left and the supremum on the right proves the estimate ``$\ge$''.

Now assume that $\|A+\pc{\Xb}\|>\sup_{A_g\in\opsp(A)}\|A_g\|=:N_A$ holds. Then there is an $\eps>0$ with
$\|A+\pc{\Xb}\|>N_A+\eps$. We conclude that $\|AQ_m\|=\|A-AP_m\|\geq \|A+\pc{\Xb}\|>N_A+\eps$ for every $m\in\Nb$.
From Corollary \ref{CLoc} we get an $n\in\Nb$ such that $\vertiii{AQ_m}_{2n+1} > N_A+\eps/2$ for every $m$. In particular, we get $k_1,k_2,...\in\Zb^N$ such that, in the notation $P_{n,k}=V_kP_nV_{-k}$ of the proof of Proposition \ref{PLoc}, $N_A+\eps/2<\|(AQ_m)P_{n,k_m}\|\le\|AP_{n,k_m}\|$ for every $m$.
Now pass to a subsequence $g=(g_j)$ of the (unbounded) sequence $(k_1,k_2,...)$ for which the limit operator $A_g$ exists.
Then
\[
N_A+\eps/2 < \|A P_{n,g_j}\| = \|V_{-g_j} A V_{g_j}P_n\| \to \|A_gP_n\| \le \|A_g\|\le N_A,\qquad j\to\infty
\]
is a contradiction.

It remains to show that $N_A$ exists as a maximum. The argument is very similar to that in the proof of
\cite[Theorem~8]{BigQuest}, where it is explained in more detail (also see Figure 1 and Remark 9 in \cite{BigQuest}).
We consider the numbers $\gamma_n:=2^{-n}$ and
\[r_l := \sum_{n=l}^\infty \gamma_n = 2^{-l+1}.\]
Then $(r_l)$ is a strictly decreasing sequence of positive numbers which tends to $0$.
From the above corollary we obtain a sequence $(D_l)\subset\Nb$ of even numbers such
that for every $l\in\Nb$
\[D_{l+1} > 2 D_l\quad\text{and}\quad
\vertiii{B\chi_FI}_{D_l} > \|B\chi_FI\| - \gamma_l \quad
\text{for every } B\in\{A\}\cup\opsp(A)\text{ and every } F\subset\Zb^N.\]

Choose a sequence $(B_n)\subset\opsp(A)$ such that
$\|B_n\|\to\sup\{\|A_g\|:A_g\in\opsp(A)\}$ as $n\to\infty$.
For each $n\in\Nb$ we are going to construct a suitably shifted copy $C_n\in\opsp(A)$
of $B_n$ as follows:

We start with an $x_n^0\in\Xb$, $\|x_n^0\|=1$,
$\diam\supp x_n^0\leq D_n$ such that  $\|B_nx_n^0\| \geq \|B_n\| - \gamma_n$. We choose a
shift $j_n^0\in\Zb^N$ which centralizes $y_n^0:=V_{j_n^0}x_n^0$ such that
$y_n^0= P_{D_n/2}y_n^0$, and define the copy
$C_n^0:= V_{j_n^0} B_n V_{-j_n^0}\in\opsp(A)$. Then we have
$\|B_n\| \geq \|C_n^0P_{D_n/2}\| \geq  \|B_n\| - \gamma_n$.

Now, for $k=1,\ldots,n$, we gradually perform a fine tuning by choosing 
$x_n^k\in\im P_{D_{n-(k-1)}/2}$,
$\|x_n^k\|=1$, $\diam\supp x_n^k\leq D_{n-k}$ such that
 $\|C_n^{k-1}P_{D_{n-(k-1)}/2}x_n^k\| \geq \|C_n^{k-1}P_{D_{n-(k-1)}/2}\| - \gamma_{n-k}$,
passing to a centralized $y_n^k:=V_{j_n^k}x_n^k$ via a shift
$j_n^k\in\{-D_{n-(k-1)}/2,\ldots,D_{n-(k-1)}/2\}^N$ and defining
$C_n^k:= V_{j_n^k} C_n^{k-1} V_{-j_n^k}\in\opsp(A)$. For this we observe
$\|C_n^kP_{D_{n-k}/2}\| \geq \|C_n^{k-1}P_{D_{n-(k-1)}/2}\| - \gamma_{n-k}$.
In particular, for $n>l\geq 1$,
the estimates $\|C_n^{n-l}P_{D_l/2}\| \geq \|B_n\| - \sum_{k = l}^n \gamma_k \geq \|B_n\| - r_l$ hold.
Finally, we define $C_n:=C_n^n$ and notice that
$C_n=V_{j_n^n+\ldots+j_n^{n-l+1}} C_n^{n-l} V_{-(j_n^n+\ldots+j_n^{n-l+1})}$,
where $|j_n^n+\ldots+j_n^{n-l+1}| \leq D_l$ by construction, thus
$\|C_nP_{2D_l}\| \geq \|C_n^{n-l}P_{D_l/2}\| \geq \|B_n\| - r_l$.

By this construction we have obtained a sequence $(C_n)\subset \opsp(A)$ 
of limit operators $C_n$ which have adjusted local norms and
such that still $\|C_n\|\to\sup\{\|A_g\|:A_g\in\opsp(A)\}$ as $n\to\infty$ holds.
By \cite[Prop. 3.104]{Marko} we can pass to a subsequence $(C_{h_n})$ of $(C_n)$ with $\Pc$-strong limit
$C\in\opsp(A)$. Then
\[
\|C\| \geq \|CP_{2D_l}\| = \lim_{n\to\infty}\|C_{h_n}P_{2D_l}\| \geq  \lim_{n\to\infty} \|B_{h_n}\| - r_l = N_A - r_l
\]
for every $l$. Since $r_l$ goes to $0$ as $l\to\infty$ the assertion follows.
\end{proof}

\begin{rem}\label{RCounterEx}
In $\Lc(\Xb,\Pc)$ the Equality \eqref{EEssNorm} does not hold in general.
\begin{itemize}
\item Consider $X:= L^p[0,1]$, and the multiplication operators $a_kI\in\lb{X}$ with
			$a_k(x):=\sin(2\pi kx)$. Then the diagonal operator
			$A:=\diag(\ldots,0,0,a_1I,a_2I,a_3I,\ldots)$ on $\Xb$ has
			operator spectrum $\{0\}$, but essential norm $1$.
\item Consider the $n\times n$ matrices
			\[B_n:=\frac{1}{n}\begin{pmatrix}1 & \cdots & 1 \\ \vdots & & \vdots \\ 1 & \cdots & 1\end{pmatrix}\]
			and the block diagonal operator $A:=\diag(\ldots,0,0,B_1,B_2,B_3,\ldots)$
			on $l^p(\Zb,\Cb)$, $1<p<\infty$.
			Then $A$ has operator spectrum $\{0\}$, but essential norm $1$
			(see \cite[Example 14]{BigQuest}).
\end{itemize}
The first example is banded but not rich, whereas the second one is
rich but not band-dominated.
Note that in the extremal cases, $p\in\{0,1,\infty\}$, the latter cannot happen since
rich operators $A\in\pb{\Xb}$ are automatically band-dominated then, by \cite[Theorem 15]{BigQuest}.
\end{rem}

Now we combine Equations \eqref{EEssNorm} and \eqref{ERegOpSpec}:

\begin{cor}\label{CEssNorm}
Let $A\in\Alpr$ be $\Pc$-Fredholm, and $B$ be a $\Pc$-regularizer. Then
\begin{equation}\label{EEssNorm2}
\|(A+\pc{\Xb})^{-1}\|=\|B+\pc{\Xb}\|=\max_{B_h\in\opsp(B)}\|B_h\|
=\max_{A_h\in\opsp(A)}\|A_h^{-1}\|.
\end{equation}
If $A\in\Alpr$ is not $\Pc$-Fredholm, then both, the RHS and the LHS of \eqref{EEssNorm2}
are infinite.
\end{cor}
\begin{proof}
The operator $B$ is band-dominated by Theorem \ref{TBDOInvCl} and rich by Theorem
\ref{TRich}. Hence Theorem \ref{TEssNorm} applies and Equation \eqref{EEssNorm}
together with Equation \eqref{ERegOpSpec} from Theorem \ref{TLimOps} provide
\eqref{EEssNorm2}. The last sentence follows from Theorem \ref{TBdORichHS}.
\end{proof}

What comes as a simple corollary here is in fact a cornerstone for large parts of the subsequent results.
Remember Theorem \ref{TBdORichHS} for $A\in\Alpr$. It says that
\[
\|(A+\pc{\Xb})^{-1}\|<\infty\qquad\textrm{if and only if}\qquad
\sup_{A_h\in\opsp(A)}\|A_h^{-1}\|<\infty.
\]
Now Corollary \ref{CEssNorm} goes far beyond: It shows that both quantities are
always equal and that the supremum is actually attained as a maximum.
\medskip

Before we continue to look at Equality \eqref{EEssNorm2} and its
ingredients from different angles, we will prove the following lemma that
will be helpful in several places but is also of interest in its own right:

\begin{lem}\label{LInverses}
Let $\Pc$ be a uniform approximate identity on $\Yb$ and $A\in\pb{\Yb}$.
Then
\begin{enumerate}
	\item[a)] The set $\Yb_0:=\{y\in\Yb:\|Q_ny\|\to 0\text{ as }n\to\infty\}$ is a closed subspace
of $\Yb$. The restriction $A_0:=A|_{\Yb_0}$ of $A$ to $\Yb_0$ belongs to $\lb{\Yb_0}$, $\|A_0\|=\|A\|$, and $\nu(A)=\nu(A_0)$.
	\item[b)] The restriction  $(A^*)_0:=A^*|_{(\Yb^*)_0}$ of $A^*$ to the (analogously defined) subspace $(\Yb^*)_0$ belongs to $\lb{(\Yb^*)_0}$ and $\|(A^*)_0\|=\|A^*\|$.
	\item[c)] If $A$ is  invertible then $A_0$ is invertible with inverse $(A_0)^{-1}=(A^{-1})_0\in\lb{\Yb_0}$ and $\|(A^{-1})_0\|=\|A^{-1}\|$. Further, $(A^*)_0$ is invertible in $\lb{(\Yb^*)_0}$ with inverse $((A^*)_0)^{-1}=((A^*)^{-1})_0=((A^{-1})^*)_0$ and $\|((A^*)^{-1})_0\|=\|(A^*)^{-1}\|=\|A^{-1}\|$.
\end{enumerate}
\end{lem}

\begin{proof}
{\it a)} It is easily checked that $\Yb_0$ is a closed subspace of $\Yb$.
$A_0(\Yb_0)\subset\Yb_0$ is from \cite[Lemma 1.1.20]{LimOps} or
\cite[Proposition 1.18.1]{SeSi3} and the formula on the norm is
\cite[Proposition 1.18.2]{SeSi3}. The inequality $\nu(A)\le\nu(A_0)$ is trivial
and it remains to prove $\nu(A)\geq \nu(A_0)$. We apply the sequence $(F_n)$
given by Proposition \ref{PFn} to obtain
\[\|Ax\|=\|F_n\|\|Ax\|\geq\|F_nAx\|\geq\|AF_nx\|-\|[A,F_n]\|\|x\|
\geq\nu(A_0)\|F_nx\|-\|[A,F_n]\|\|x\|\]
for every $x\in\Yb$ and every $n\in\Nb$. Sending $n\to\infty$ we get
$\|Ax\|\geq\nu(A_0)\|x\|$ for every $x\in\Yb$, and taking the infimum over all
$\|x\|=1$ we finally arrive at $\nu(A)\geq\nu(A_0)$.

{\it b)} The inclusion $(A^*)_0((\Yb^*)_0)\subset(\Yb^*)_0$ follows by the same means.
Here in this dual setting \cite[Proposition 1.18.2]{SeSi3} may not
be applicable anymore since $\Pc^*$ is not necessarily subject to $(\Pc3)$.
Therefore we need another proof for the formula on the norms.

Let $\eps>0$ and choose $y\in\Yb$, $\|y\|=1$
such that $\|A\|\leq \|Ay\|+\eps$. Since $\Pc$ is an approximate identity we
find a $k$ such that  $\|A\|\leq \|P_kAy\|+2\eps$. Now, by Hahn Banach there
is a functional $g_0$ on $\im P_k$, $\|g_0\|=1$, with $\|P_kAy\|=|g_0(P_kAy)|$.
Thus, setting $g:=g_0\circ P_k$ we obtain a functional $g\in\Yb^*$ , $\|g\|=1$,
such that actually $g\in(\Yb^*)_0$ with norm $1$, hence
\[\|(A^*)_0\|\leq\|A^*\|=\|A\|\leq |g(Ay)|+2\eps\leq\|A^*g\|+2\eps\leq\|(A^*)_0\|+2\eps.\]
Since $\eps>0$ is arbitrary this shows $\|(A^*)_0\|=\|A^*\|$.

{\it c)} Let $A$ be invertible. Then $A^*$ is invertible as well where $(A^*)^{-1}=(A^{-1})^*$.
The invertibility of $A_0$ and $(A^*)_0$ as well as the formulas for their inverses
follow from \cite[Corollary 1.9, Corollary 1.19]{SeSi3}.
Since $A^{-1}$ is still in $\pb{\Yb}$ by Theorem \ref{TPFredh} we can apply the already
proved formulas on the norms also to $A^{-1}$.
\end{proof}

Lemma \ref{LInverses} enables us to restrict consideration to elements $x\in\Yb_0$
when approximating $\|A\|$ or $\nu(A)$ by $\|Ax\|$ -- and similarly for
$A^*$, $A^{-1}$ or $(A^*)^{-1}$ in place of $A$. In combination with
$\Pc$-convergence this turns out to be a lot more convenient than having
to work with $x\in\Yb$.
The proof of the following proposition shows what we mean by that:

\begin{prop}\label{prop:lown-limop}
Let $A\in\pb{\Xb}$ and $A_h\in\opsp(A)$. Then $\nu(A_h)\ge\nu(A)$.
\end{prop}
In a sense, this result is a lower counterpart of the norm inequality from Proposition \ref{prop:limops}.
Together they show that $\nu(A)\le\nu(A_h)\le\|A_h\|\le\|A\|$.
\begin{proof}
Let $\eps>0$ and let $h=(h_n)$ be a sequence in $\Zb^N$ with $V_{-h_n}AV_{h_n}\pto A_h$.
By closedness of $\pb{\Xb}$ under $\Pc$-strong convergence (the first part of \eqref{eq:liminf}),
also $A_h\in\pb{\Xb}$.
We apply Lemma~\ref{LInverses}~a) to $A_h$.
There is a $x_0\in\Xb_0$ with $\|x_0\|=1$ such that
$\nu(A_h)=\nu((A_h)_0)>\|A_hx_0\|-\eps$. Now truncate $x_0$ and renormalize.
Since $x_0\in\Xb_0$, one has $\|P_kx_0\|^{-1}P_kx_0\to x_0$ as $k\to\infty$.
So, for sufficiently large $k\in\Nb$, $x:=\|P_kx_0\|^{-1}P_kx_0$ also fulfills
$\nu(A_h)>\|A_hx\|-\eps$, where $\|x\|=1$ and now $x=P_kx$.
Choose $n\in\Nb$ large enough that $\|(A_h-V_{-h_n}AV_{h_n})P_k\|<\eps$
and conclude that
\[
\nu(A_h)\ge\|A_hx\|-\eps=\|A_hP_kx\|-\eps\ge\|V_{-h_n}A\underbrace{V_{h_n}P_kx}_{=:x_n}\|-2\eps
=\|Ax_n\|-2\eps\ge\nu(A)\|x_n\|-2\eps.
\]
But since $\|x_n\|=\|P_kx\|=\|x\|=1$ and $\eps>0$ is arbitrary, we are finished.
\end{proof}

\section{The \texorpdfstring{$\Pc$}{P}-essential pseudospectrum} \label{sec:essps}
With our formula \eqref{EEssNorm2} it is possible to study resolvent norms in $\Alpr/\pc{\Xb}$.
To do this replace $A$ by $A-\lambda I$ in \eqref{EEssNorm2} and recall that
$(A-\lambda I)_h=A_h-\lambda I$. Then \eqref{EEssNorm2} turns into \eqref{eq:resolvents}.
This motivates to study the following kind of pseudospectra:

\begin{defn}
For $A\in\pb{\Xb}$ and $\eps>0$, the {\sl $\Pc$-essential $\eps$-pseudospectrum} is
defined as
\[\speess(A):=\spe(A+\pc{\Xb}):=\{\lambda\in\Cb:\|(A-\lambda I+\pc{\Xb})^{-1}\|> 1/\eps\}.\]
Recall that, in contrast, the $\Pc$-essential spectrum of $A$ is
\begin{eqnarray*}
\spess(A)=\spc(A+\pc{\Xb})&=&\{\lambda\in\Cb:A-\lambda I\text{ is not $\Pc$-Fredholm}\}\\
&=&\{\lambda\in\Cb:\|(A-\lambda I+\pc{\Xb})^{-1}\|=\infty\}.
\end{eqnarray*}
\end{defn}

\begin{rem}
Recall that in case $\dim X<\infty$ every $\Pc$-compact operator is compact, hence
every $\Pc$-Fredholm operator also Fredholm. By Proposition \ref{PFredh} every
Fredholm operator $A\in\pb{\Xb}$ is $\Pc$-Fredholm, thus we can conclude that
for all $A\in\pb{\Xb}$, $\Xb=l^p(\Zb^N,X)$ with $\dim X<\infty$, the $\Pc$-essential
spectrum and the (classical) essential spectrum coincide:
\[\spess(A)=\spc(A+\lc{\Xb})=\{\lambda\in\Cb:A-\lambda I\text{ is not Fredholm}\}.\]
We will address this case again in more detail in Section \ref{SFinDim}.
\end{rem}

Now here is our immediate consequence of \eqref{EEssNorm2}:

\begin{thm} \label{Tspeess}
Let $A\in\Alpr$ and $\eps>0$. Then
\begin{equation}\label{EEssSp}
\speess(A)=\bigcup_{A_h\in\opsp(A)}\spe(A_h).
\end{equation}
\end{thm}
\begin{proof}
Equation \eqref{eq:resolvents} clearly implies \eqref{EEssSp}
since $\max_{A_h\in\opsp(A)}\|(A_h-\lambda I)^{-1}\|>1/\eps$ if and only if
$\lambda\in\cup_{A_h\in\opsp(A)}\spe(A_h)$. The particular case where
$\max_{A_h}\|(A_h-\lambda I)^{-1}\|=\infty$ corresponds to
$\lambda\in\cup_{A_h}\spc(A_h)$, see \eqref{eq:specess}.
\end{proof}

With \eqref{EEssSp} we have arrived at an $\eps$-version \eqref{eq:specesseps}
of \eqref{eq:specess}, the second part of (v) in the introduction.
It is known that it may be easier to compute pseudospectra of limit operators
than their spectra. So, numerically, computing $\speess(A)$ via \eqref{EEssSp}
aka \eqref{eq:specesseps}
is in general simpler than computing $\spess(A)$ via \eqref{eq:specess}.
In the end, one is probably interested in $\spess(A)$. The good news is that
this can be approximated by $\speess(A)$ as $\eps\to 0$.
It is a standard result that the $\eps$-pseudospectra converge to the spectrum
as $\eps\to 0$. For the reader's convenience, we state and prove the result
here for our concrete setting of $\Pc$-essential (pseudo)spectra:

\begin{prop} \label{prop:Hausdorff}
For every $A\in\pb{\Xb}$, the sets $\overline{\speess(A)}$ converge\footnote{We
consider the closure of $\speess(A)$ since the Hausdorff metric is defined for
compact sets only.} to $\spess(A)$ w.r.t.~the Hausdorff metric as $\eps\to 0$.
\end{prop}

\begin{proof}
Clearly, $\spess(A)\subset\speess(A)\subset\overline{\speess(A)}\subset\spc_{\delta,\ess}(A)$
for all $0<\eps<\delta$.
On the other hand, assume that there is a sequence $(\lambda_n)$ of points
$\lambda_n\in \spc_{1/n,\ess}(A)$ which stay bounded away from the $\Pc$-essential
spectrum. By a simple Neumann series argument $(\lambda_n)$ is bounded, hence it
has a convergent subsequence. Without loss of generality let already $(\lambda_n)$
converge to $\lambda$. Since the norms $\|(A-\lambda_n I+\pc{\Xb})^{-1}\|> n$
tend to infinity, we find that $A-\lambda I+\pc{\Xb}$ cannot be invertible in
$\pb{\Xb}/\pc{\Xb}$, that is $\lambda\in\spess(A)$, a contradiction.
\end{proof}

From Theorem \ref{Tspeess} and Proposition \ref{prop:Hausdorff} we get
the following corollary:

\begin{cor}\label{CEssSp}
Let $A\in\Alpr$. Then
\[
\spess(A)
=\lim_{\eps\to0}\overline{\bigcup_{A_h\in\opsp(A)}\spe(A_h)}
=\bigcap_{\eps>0}\bigcup_{A_h\in\opsp(A)}\spe(A_h).
\]
\end{cor}

\begin{rem}
{\bf a)} Note that Corollary \ref{CEssSp}, although derived via our new Equations
\eqref{EEssNorm2} and \eqref{EEssSp}, in fact says nothing more than Theorem \ref{TBdORichHS}
and Equation \eqref{eq:specess}.

{\bf b)} Several authors define pseudospectra with ``$\geq 1/\eps$'' instead of
``$> 1/\eps$'', which leads to compact pseudospectra, but sometimes causes additional
difficulties. (For example, the analogue of Proposition \ref{prop:speps_perturb} below
is no longer true in arbitrary Banach space $Y$ if ``$> 1/\eps$'' is replaced by ``$\ge 1/\eps$''
in the definition of $\spe(A)$ and if the union below is taken over all $\|K\|\le\eps$ instead of
all $\|K\|<\eps$, cf. \cite{SharPsSp}.)
Anyway, our preceding results hold for both definitions.

{\bf c)} Similar observations are to be expected for $(N,\eps)$-pseudospectra as well.
\end{rem}

Another well-known and very useful characterization of pseudospectra of operators $A$
is given as the union of spectra of small perturbations of $A$.

\begin{prop}\label{prop:speps_perturb}(Cf. \cite[Section 7.1]{BandToep})
Let $Y$ be a Banach space, $A\in\lb{Y}$ and $\eps>0$. Then
\[\spe(A)=\bigcup_{\|K\|<\eps}\spc(A+K)=\bigcup_{\|K\|<\eps,\,\rk K\le1}\spc(A+K).\]
\end{prop}

In the following Proposition we improve this result in case of $A\in\pb{\Xb}$.

\begin{prop}\label{PPerturb}
Let $\Cc\subset\pb{\Xb}$ be an algebra containing all rank-$1$-operators with only
finitely many non-zero entries in the respective matrix representation, let $A\in\Cc$ 	
and let $\eps>0$. Then
\[
\spe(A)=\bigcup_{\|K\|<\eps}\spc(A+K)=\bigcup_{\substack{\|K\|<\eps,\\K\in\Cc}}\spc(A+K)
=\bigcup_{\substack{\|K\|<\eps,\\K\in\pc{\Xb}}}\spc(A+K)
=\bigcup_{\substack{\|K\|<\eps,\\K\in\pc{\Xb}\cap\Cc,\\ \rk K\le1}}\spc(A+K).
\]
\end{prop}

\begin{proof}
Abbreviate the sets in this claim from left to right by $S_1,...,S_5$.
$S_1=S_2$ holds by the previous proposition, $S_2\supset S_3\supset S_5$ and
$S_2\supset S_4\supset S_5$ are obvious. Thus, it remains to prove $S_5\supset S_1$.

So let $\lambda\in S_1$. Since the case $\lambda\in\spc(A)$ is clear,
let $B:=A-\lambda I$ be invertible with $\|B^{-1}\|>1/\eps$.
By Lemma \ref{LInverses}, also $B_0:=B|_{\Xb_0}$ is invertible and
$\|(B_0)^{-1}\|=\|B^{-1}\|>1/\eps$, so that
there exists an $x_0\in\Xb_0$, $\|x_0\|=1$, with
$\|Bx_0\|=\|B_0x_0\|< \eps$. As in the proof of Proposition \ref{prop:lown-limop},
take $k$ sufficiently large that also
$x:=\|P_kx_0\|^{-1}P_kx_0$ fulfills $\|Bx\|<\eps$, where $\|x\|=1$ and $P_kx=x$.
By the Hahn-Banach Theorem there exists a functional $\varphi$ with
$\|\varphi\|=\varphi(x)=1$ and $\varphi\circ P_k=\varphi$. Now, we define
$\tilde{K}u:=-\varphi(u)x$ and $Ku:=-\varphi(u)Bx$ for every $u\in\Xb$.
Then $\tilde{K}$, $K$ have rank $1$ and $\|K\|\leq\|\varphi\|\|Bx\|<\eps$.
Moreover, both $\tilde{K}=P_k\tilde{K}P_k$ and $K=B\tilde{K}$ belong to
$\pc{\Xb}\cap\Cc$. Finally, with $(B+K)x=Bx-\varphi(x)Bx=0$, we summarize:
$\lambda\in\spc(A+K),\, \|K\|<\eps,\,K\in\pc{\Xb}\cap\Cc,\,\rk K=1.$
\end{proof}

Also for the $\Pc$-essential pseudospectra for classes of rich band-dominated
operators we can obtain a characterization via perturbations.

\begin{thm}
Let $\Cc$ be one of the algebras of all rich band operators or all rich band-dominated
operators\footnote{Actually, one can consider many more subalgebras $\Cc$ of $\Alpr$
as long as one can define operators of the form \eqref{ET} there. Another example is
the set of all rich operators in the Wiener algebra, see e.g. \cite{RochFredTh} or
\cite[\S3.7.3]{HabilMarko}.}
on $\Xb$ and let $A\in\Cc$. For $\eps>0$
\[\speess(A)=\bigcup_{\substack{\|T\|<\eps,\\T\in\pb{\Xb}}}\spess(A+T)
=\bigcup_{\substack{\|T\|<\eps,\\T\in\Cc}}\spess(A+T).\]
\end{thm}
\begin{proof}
For each $L\in\pb{\Xb}$, abbreviate the coset $L+\pc{\Xb}\in\pb{\Xb}/\pc{\Xb}$ by $L^\circ$.
Now let $A\in\Cc$ and $\lambda\not\in\speess(A)$. With $B:=A-\lambda I$,
the coset $B^\circ$ is invertible and $\|(B^\circ)^{-1}\|\leq 1/\eps$.
For arbitrary $T\in\pb{\Xb}$ with $\|T\|<\eps$, one has
$\|(B^\circ)^{-1}T^\circ\|<1$, so that $I^\circ+(B^\circ)^{-1}T^\circ$ is invertible.
Thus, $(B+T)^\circ=B^\circ(I^\circ+(B^\circ)^{-1}T^\circ)$ is invertible, whence
$\lambda\not\in\spess(A+T)$.
Together with Theorem \ref{Tspeess} we conclude the following inclusions:
\begin{align*}
\bigcup_{\substack{\|T\|<\eps,\\T\in\Cc}}\spess(A+T) & \subset
\bigcup_{\substack{\|T\|<\eps,\\T\in\pb{\Xb}}}\spess(A+T)\subset
\speess(A)=\bigcup_{A_h\in\opsp(A)}\spe(A_h).
\end{align*}
It remains to show that the right-most set is contained in the left-most.
So let $A_h\in\opsp(A)$ and $\lambda\in\spe(A_h)$. By Proposition \ref{PPerturb},
$\lambda\in\spc(A_h+K)$ for some $K\in\pc{\Xb}\cap\Cc$ with $\|K\|<\eps$.
Now choose a subsequence $g$ of $h$ such that
all cubes $g_n+\{-n,...,n\}^N$ are pairwise disjoint,
and define
\begin{equation}\label{ET}
T:=\sum_{n\in\Nb}V_{g_n}P_nKP_nV_{-g_n}.
\end{equation}
$T$ is a well-defined block-diagonal operator\footnote{By our assumption on $g$, the blocks
$V_{g_n}P_nKP_nV_{-g_n}$ do not overlap.} belonging to $\Cc$ with $\|T\|\leq\|K\|<\eps$
and $T_g=K$. Since
\[
(A-\lambda I+T)_g=A_g-\lambda I +T_g=A_h-\lambda I + K,
\]
we find that $\lambda\in\spc(A_h+K)=\spc((A+T)_g)$, whence $\lambda\in\spess(A+T)$ by
\eqref{eq:specess}.
\end{proof}

\begin{rem}
The above proof that the pseudospectrum is a superset of the union of spectra of
perturbations works in every Banach algebra. In $C^*$-algebras also the converse
is true, although it may fail in the general case. For more details see e.g.
\cite[Page 121]{Standard}.
\end{rem}

\section{The \texorpdfstring{$\Pc$}{P}-essential lower norm}  \label{sec:esslowernorm}
Let $A\in\pb{\Xb}$ be $\Pc$-Fredholm. By Theorem \ref{TLimOps} and
the last point of Lemma \ref{lem:bddbelow}, we can rewrite the right-hand
side of \eqref{EEssNorm2} in terms of lower norms of the limit operators:
\begin{equation}\label{ENormInvLowerN}
\max_{A_h\in\opsp(A)}\|A_h^{-1}\|
	=\left(\min_{A_h\in\opsp(A)}\nu(A_h)\right)^{-1}.
\end{equation}
Our aim for this section is to present alternative valuable characterizations 
of the essential norm $\|(A+\pc{\Xb})^{-1}\|$ on the left-hand side of \eqref{EEssNorm2} 
in terms of lower norms of (perturbations and restrictions of) the operator $A$
directly, which do not count on limit operators.

\subsection{1st approach: Lower norms of asymptotic compressions}
We start again with the abstract setting of a Banach space $\Yb$ with a uniform 
approximate projection $\Pc=(P_n)$, and we make the following simple observation:

\begin{lem}\label{LMuLim}
For $A\in\lb{\Yb}$,
$\displaystyle\lim_{m\to\infty}\nu(A|_{\im Q_m})=\sup_{m\in\Nb}\nu(A|_{\im Q_m})$,
where $A|_{\im Q_m}:\im Q_m\to\Yb$.
\end{lem}
\begin{proof}
The sequence of compressions is bounded by $\nu(A|_{\im Q_m})\le\|A\|$.
Convergence to the supremum follows from the monotonicity
$\nu(A|_{\im Q_{m+1}})\geq \nu(A|_{\im Q_m})$ since
$\im Q_{m+1}\subset\im Q_m$.
\end{proof}

\begin{defn}
For $A\in\lb{\Yb}$ set
\[
\tilde{\mu}(A):= \lim_{m\to\infty}\nu(A|_{\im Q_m}),\quad
\mu(A):=\min\{\tilde{\mu}(A),\,\tilde{\mu}(A^*)\}.
\]
\end{defn}

In Section \ref{SHilbert} we will see that in the case of appropriate Hilbert spaces $\Yb$
this $\mu(A)$ serves as a characterization for the essential norm $\|(A+\pc{\Yb})^{-1}\|$
for every operator $A\in\pb{\Yb}$ (cf. Theorem \ref{TPLowNormHilb}).
However, beyond the comfortable Hilbert space case we are still able to prove this
observation for all rich band-dominated operators on all $\Xb$.

\begin{thm}\label{TPLowNorm}
Let $A\in\Alpr$. Then\footnote{We again use the notation $\|b^{-1}\|^{-1}=0$ for 
non-invertible elements $b$.}
\begin{equation}\label{EPLowNorm}
\|(A+\pc{\Xb})^{-1}\|^{-1}=\mu(A).
\end{equation}
\end{thm}

Before we start with the proof, we want to make the following remark. An equivalent way of saying that $\nu(A)=0$ is that there exists a so-called Weyl sequence of $A$, that is a sequence $(x_n)$ of elements $x_n \in \Yb$ with $\|x_n\|=1$ for all $n \in \Nb$, such that $\|Ax_n\|\to 0$ as $n \to \infty$. So $A$ is invertible iff neither $A$ nor $A^*$ has a Weyl sequence (cf. Lemma \ref{lem:bddbelow}). Moreover, $A$ is not even Fredholm if it has a weak Weyl sequence, where the latter refers to a Weyl sequence $(x_n)$ that weakly converges to zero (see e.g. \cite[Lemma 4.3.15]{Davies}). Similarly, we call a Weyl sequence $(x_n)$ a $\Pc$-Weyl sequence\footnote{In \cite{HiSi} a continuous analogue of this concept is mentioned and denoted as Zhislin sequence.} if additionally (instead of weak convergence) $\|P_mx_n\| \to 0$ as $n \to \infty$ for every fixed $m \in \Nb$. Then we have the following:

\newpage 

\begin{lem}
Let $A \in \Lc(\Yb)$. Then $\tilde{\mu}(A)=0$ iff $A$ has a $\Pc$-Weyl sequence.
\end{lem}

\begin{proof}
If $\tilde{\mu}(A) = 0$, then there exists a sequence $(x_n)$ of elements $x_n \in \Yb$ with $\|x_n\| = 1$ such that $x_n \in \im Q_n$ and $Ax_n \to 0$ as $n \to \infty$. This obviously defines a $\Pc$-Weyl sequence.

Conversely let $(x_n)$ be a $\Pc$-Weyl sequence of $A$. Then for every $m \in \Nb$ there exists $n \in \Nb$ such that $\|P_mx_n\| < \frac{1}{m}$ and $\|Ax_n\| < \frac{1}{m}$. This implies
\[\frac{\|AQ_mx_n\|}{\|Q_mx_n\|} = \frac{\|Ax_n - AP_mx_n\|}{\|x_n-P_mx_n\|} < \frac{\frac{1}{m} + \|A\|\frac{1}{m}}{1 - \frac{1}{m}} = \frac{1 + \|A\|}{m-1}.\]
Hence $\nu(A|_{\im Q_m}) \to 0$ as $m \to \infty$.
\end{proof}

Thus $\mu(A)=0$ iff $A$ or $A^*$ has a $\Pc$-Weyl sequence. Consequently, Theorem \ref{TPLowNorm} and further theorems relating $\mu(A)=0$ to non-$\Pc$-Fredholmness of $A$ characterize the latter in terms of $\Pc$-Weyl sequences. So this is a further instance that generalizes from Fredholmness to $\Pc$-Fredholmness.

The proof of Theorem \ref{TPLowNorm} is a simple consequence of the following lemmas:

\begin{lem}\label{LPLowNorm1}
Let $\Pc$ be a uniform approximate projection on the Banach space $\Yb$ and $A\in\pb{\Yb}$. Then $\|(A+\pc{\Yb})^{-1}\|^{-1}\leq\min\{\tilde{\mu}(A),\tilde{\mu}(A^*)\}.$
If $A$ is $\Pc$-Fredholm, then it even holds that $\|(A+\pc{\Yb})^{-1}\|^{-1}=\tilde{\mu}(A)=\tilde{\mu}(A^*)=\mu(A)$.
\end{lem}

\begin{proof}
There is nothing to prove if $A$ is not $\Pc$-Fredholm, since the LHS equals zero in this case.
If $A$ is $\Pc$-Fredholm let $\eps>0$ be arbitrary and choose $B_0\in (A+\pc{\Yb})^{-1}$.
Since $B_0A-I=:K\in\pc{\Yb}$ we get for all sufficiently large $m$ that
$\|Q_mB_0AQ_m-Q_m\|=\|Q_mKQ_m\|$ is small enough that
$Q_mB_0AQ_m=Q_m+Q_mKQ_m$ is invertible in $\lb{\im Q_m}$ with
\[
\underbrace{Q_m(Q_mB_0AQ_m)^{-1}Q_mB_0}_{\displaystyle \quad \quad
\quad \quad \quad =:B_1\in\lb{\Yb}}AQ_m=Q_m
\quad\text{and}\quad \|Q_mB_0-B_1\|<\eps,
\]
and that $\|Q_mB_0\|\le \|(A+\pc{\Yb})^{-1}\|+\eps$, taking Proposition \ref{PEssNorm} into account.
By this and Lemma \ref{lem:bddbelow} we get that $\nu(A|_{\im Q_m})>0$, hence the compression
$A|_{\im Q_m}:\im Q_m \to\im AQ_m$ is invertible and the compression
$B_1|_{\im AQ_m}:\im AQ_m \to\im Q_m$ is its (unique) inverse.
We conclude that for sufficiently large $m$
\[
\nu(A|_{\im Q_m})^{-1}=\|B_1|_{\im AQ_m}\|\leq \|B_1\|\leq
	\|Q_mB_0\|+\|B_1-Q_mB_0\|\leq \|(A+\pc{\Yb})^{-1}\|+2\eps.
\]

On the other hand, $AQ_m$ is $\Pc$-Fredholm and thus has a $\Pc$-regularizer $C$.
So $\|(AQ_mC-I)Q_k\|<\delta:=\eps/(2\|B_1\|)$ if $k$ is large enough.
Moreover, from $B_1AQ_m=Q_m$ and $Q_m\cong I$ modulo $\pc{\Yb}$ we get that $B_1$
and hence also $B_1Q_k$ is inverse to $A$ modulo $\pc{\Yb}$. Consequently,
\begin{align*}
\|(A+\pc{\Yb})^{-1}\|&=\|B_1Q_k+\pc{\Yb}\|\le\|B_1Q_k\|\le\|B_1AQ_mCQ_k\|+\|B_1\|\|(AQ_mC-I)Q_k\|\\
&<\|B_1|_{\im AQ_m}\|\|AQ_mCQ_k\|+\|B_1\|\delta < \|B_1|_{\im AQ_m}\|(\|Q_k\|+\delta)+\|B_1\|\delta\\
&\le\|B_1|_{\im AQ_m}\|+2\|B_1\|\delta=\|B_1|_{\im AQ_m}\|+\eps=\nu(A|_{\im Q_m})^{-1}+\eps.
\end{align*}
Since $\eps>0$ is arbitrary, we arrive at
$\tilde{\mu}(A)=\|(A+\pc{\Yb})^{-1}\|^{-1}$, by Lemma \ref{LMuLim}.
By the same observation for $A^*\in\lb{\Yb^*,\Pc^*}$ we find
$\tilde{\mu}(A^*)=\|(A^*+\lc{\Yb^*,\Pc^*})^{-1}\|^{-1}=\|(A+\pc{\Yb})^{-1}\|^{-1},$
where Proposition \ref{PEssNorm} justifies the latter equality.
\end{proof}

\begin{lem}\label{LPLowNorm2}
Let $A\in\pb{\Xb}$. Then $\mu(A)\leq\inf\{\|A_h^{-1}\|^{-1}:A_h\in\opsp(A)\}.$
For any invertible $A_h\in\opsp(A)$ we even have both
$\tilde{\mu}(A),\tilde{\mu}(A^*)\leq\|A_h^{-1}\|^{-1}.$
\end{lem}

\begin{proof}
Let $A_g\in\opsp(A)$.

\textit{1st case:} $A_g$ is not invertible. For every $\eps>0$ there is
a $\Pc$-compact operator $T$ of the norm $1$ and such that $\|A_gT\|<\eps$
or $\|TA_g\|<\eps$ (cf. \cite[Theorem 11]{SeFre}). Let $m \in \Nb$. It follows from
$(Q_m)_g=I$ that
\[\|V_{-g_n}AQ_mV_{g_n}T\|<2\eps\quad\text{or}\quad
  \|TV_{-g_n}Q_mAV_{g_n}\|<2\eps\quad\text{for all sufficiently large $n$.}\]
Setting $T_n:=V_{g_n}TV_{-g_n}$ we have $\|AQ_mT_n\|<2\eps$ or
$\|T_nQ_mA\|<2\eps$. Since $\|Q_mT_n\|$ and $\|T_nQ_m\|$ tend to $1$ as
$n\to\infty$ we conclude
\[\frac{\|AQ_mT_n\|}{\|Q_mT_n\|}<3\eps\quad\text{or}\quad
  \frac{\|T_nQ_mA\|}{\|T_nQ_m\|}<3\eps\quad\text{for large $n$}.\]
This yields $\nu(A|_{\im Q_m})<3\eps$ or $\nu(A^*|_{\im Q_m^*})<3\eps$, and
since $\eps$ and $m$ are arbitrary, we conclude $\mu(A)=0$.

\textit{2nd case:} $A_g$ is invertible.
Now we proceed similarly to the proof of Proposition \ref{prop:lown-limop}.
By Lemma \ref{LInverses} the compression
$(A_g)_0$ is invertible, $((A_g)_0)^{-1}=(A_g^{-1})_0$ and
$\|(A_g^{-1})_0\|=\|A_g^{-1}\|$.
Let $\eps>0$. Then there exists an $x_0\in\Xb_0$, $\|x_0\|=1$, with
$\|A_gx_0\|=\|(A_g)_0x_0\|< \nu((A_g)_0)+\eps = \|((A_g)_0)^{-1}\|^{-1}+\eps=
\|A_g^{-1}\|^{-1}+\eps$. For sufficiently large $k$ also $x:=\|P_kx_0\|^{-1}P_kx_0$
fulfills $\|A_gx\|<\|A_g^{-1}\|^{-1}+\eps$, where $\|x\|=1$ and $P_kx=x$. For sufficiently large $n$, $\|(V_{-g_n}AQ_mV_{g_n}-A_g)P_k\|\leq\eps$ holds
and we find
\begin{align*}
\|AQ_mV_{g_n}x\|&=\|V_{-g_n}AQ_mV_{g_n}P_kx\|
\leq \|A_gP_kx\|+\eps = \|A_gx\|+\eps \leq \|A_g^{-1}\|^{-1}+2\eps,
\end{align*}
whence $\nu(A|_{\im Q_m})\leq\|A_g^{-1}\|^{-1}+2\eps$ holds for every $m$. Since
$\eps>0$ is arbitrary, $\mu(A)\leq\tilde{\mu}(A)\leq \|A_g^{-1}\|^{-1}$.

In the dual setting we proceed in exactly the same way to get
$\tilde{\mu}(A^*)\leq \|(A_g^*)^{-1}\|^{-1}=\|(A_g)^{-1}\|^{-1}$ by considering
the compressions $(A_g^*)_0$.
\end{proof}

Thus, we have for all $A\in\pb{\Xb}$ that
\begin{equation} \label{eq:squeeze}
\|(A+\pc{\Xb})^{-1}\|^{-1}\leq\mu(A)\leq\inf\{\|A_h^{-1}\|^{-1}:A_h\in\opsp(A)\}.
\end{equation}
For rich band-dominated operators the left-hand side and the right-hand side
coincide by Corollary \ref{CEssNorm}, hence Theorem \ref{TPLowNorm} follows.

\subsection{2nd approach: Lower norms of \texorpdfstring{$\Pc$}{P}-compact perturbations}

For $A\in\pb{\Yb}$ we define the $\Pc$-essential lower norm of $A$ by
\[\nuess(A)=\sup\{\nu(A+K):K\in\pc{\Yb}\}\]
and we want to study the relations between $\nuess(A)$ and $\tilde{\mu}(A)$.
\newpage 

\begin{prop} \label{PNuess}
Let $\Pc$ be a uniform approximate projection on $\Yb$ and $A\in\pb{\Yb}$.
Then $\nuess(A)\leq\tilde{\mu}(A)$ and $\nuess(A^*)\leq\tilde{\mu}(A^*)$.
If $\nu(A)>0$, then $\nuess(A)=\tilde{\mu}(A)$.
If $\nu(A^*)>0$, then $\nuess(A^*)=\tilde{\mu}(A^*)$.
\end{prop}
\begin{proof}
For $\eps>0$ let $K\in\pc{\Yb}$ be s.t. $\nu(A+K)\geq\nuess(A)-\eps$
and $m$ s.t. $\|KQ_m\|\leq\eps$. Then
\[\tilde{\mu}(A)\geq\nu(A|_{\im Q_m})\geq\nu((A+K)|_{\im Q_m})-\|KQ_m\|\geq\nu(A+K)-\eps
\geq\nuess(A)-2\eps.\]
Since $\eps$ was arbitrary, the estimate $\nuess(A)\leq\tilde{\mu}(A)$ is proved.

Now, let $A$ be bounded below and assume that there are constants $c,d$ such that
$\nuess(A)<c<d<\tilde{\mu}(A)$. By the
definition $\nu(Q_kA+\alpha P_kA)\leq\nuess(A)$ for all $k\in\Nb$ and all scalars $\alpha$.
In particular, for every $k\in\Nb$ and $\alpha>0$ there exists $\|x_{k,\alpha}\|=1$
such that $\|(Q_kA+\alpha P_kA)x_{k,\alpha}\|<c$. This further implies
$\|Q_kAx_{k,\alpha}\|<c$ and $\alpha \|P_kAx_{k,\alpha}\|<c$ by $(\Pc1)$.
Now, choose $\eps>0$ such that
$c+\eps+2\eps\|A\|/\nu(A)< d(1-2\eps/\nu(A))$,
and $\alpha>1$ such that $c/\alpha<\eps$.

Fix $n\in\Nb$, take the sequence $(F_n)$ from Proposition \ref{PFn}, and choose
$m\in\Nb$ such that $P_nF_m=P_n$ and $\|[F_m,A]\|<\eps$. Then choose
$k\in\Nb$ such that $F_mP_k=F_m$. From $\alpha\|P_kAx_{k,\alpha}\|<c$ we get
$\|F_mP_kAx_{k,\alpha}\|<c/\alpha$ and we conclude that
$\|AF_mx_{k,\alpha}\|\leq c/\alpha + \|[F_m,A]\|<2\eps$, thus
$\|P_nx_{k,\alpha}\|\leq\|F_mx_{k,\alpha}\|<2\eps/\nu(A)$ and
$\|Q_nx_{k,\alpha}\|\geq 1-2\eps/\nu(A)$. Now
\[\nu(A|_{\im Q_n})\leq\frac{\|AQ_nx_{k,\alpha}\|}{\|Q_nx_{k,\alpha}\|}
\leq\frac{\|Q_kAx_{k,\alpha}\|+\|P_kAx_{k,\alpha}\|+\|A\|\|P_nx_{k,\alpha}\|}{\|Q_nx_{k,\alpha}\|}
<\frac{c+c/\alpha+2\eps\|A\|/\nu(A)}{1-2\eps/\nu(A)}<d\]
and since $n\in\Nb$ is arbitrary it follows
$\tilde{\mu}(A)=\lim_n\nu(A|_{\im Q_n})\leq d < \tilde{\mu}(A)$,
a contradiction.

Finally, applying the already proved assertions to $A^*\in\lb{\Yb^*,\Pc^*}$ finishes
the proof.
\end{proof}
Since $\nuess$ and $\tilde{\mu}$ are invariant under $\Pc$-compact perturbations
it actually holds
\begin{cor}\label{CNuessMu}
Let $\Pc$ be a uniform approximate projection on $\Yb$ and $A\in\pb{\Yb}$.
If $A+\pc{\Yb}$ contains an operator being bounded below, then
$\nuess(B)=\tilde{\mu}(B)>0$ for all $B\in A+\pc{\Yb}$.
If $A^*+\lc{\Yb^*,\Pc^*}$ contains an operator being bounded below, then
$\nuess(B)=\tilde{\mu}(B)>0$ for all $B\in A^*+\lc{\Yb^*,\Pc^*}$.
\end{cor}

\begin{cor} \label{CNuess}
Let $\Pc$ be a uniform approximate projection on $\Yb$ and $A\in\pb{\Yb}$.
Then we have either $\nuess(A)=0$ or $\nuess(A)=\tilde{\mu}(A)>0$.
Furthermore we have either $\nuess(A^*)=0$ or $\nuess(A^*)=\tilde{\mu}(A^*)>0$.
\end{cor}

\begin{rem}
Fredholm operators with positive index on $l^2(\Zb,\Cb)$ show that
$0=\nuess(A)<\tilde{\mu}(A)$ can happen. Similarly, negative index yields
$0=\nuess(A^*)<\tilde{\mu}(A^*)$.
\end{rem}

\begin{cor}\label{CFredLX}
Let $A\in\pb{\Xb}$ and $A+\pc{\Xb}$ contain a Fredholm operator. Then
\[\max\{\nuess(A),\nuess(A^*)\}=\mu(A)=\|(A+\pc{\Xb})^{-1}\|^{-1}>0.\]
If this Fredholm operator has index $0$, then additionally $\nuess(A) = \nuess(A^*)$.
\end{cor}

\begin{proof}
W.l.o.g. let $A$ be Fredholm.
Firstly, we recall that $A$ is automatically $\Pc$-Fredholm by Proposition \ref{PFredh}), and that Lemma \ref{LPLowNorm1} applies.
With the help of shifts and projections one easily constructs a one-sided invertible band operator $S$ with banded one-sided inverse and $\ind S=-\ind A$ (cf. e.g. \cite[Lemma 24]{SeFre}). We consider the case $\ind A>0$. Then $SA$ is Fredholm of index zero and with Proposition \ref{PFredh} we find $C\in\pb{\Xb}$ invertible and $K\in\pc{\Xb}$ such that $SA=C+K$, hence $S^lC=A-S^lK\in A+\pc{\Xb}$ is right invertible with the right inverse $C^{-1}S\in\pb{\Xb}$. In the case $\ind A<0$ we proceed similarly, and in the case $\ind A=0$, we simply choose $S=I$ and get $C\in A+\pc{\Xb}$ invertible. Thus Corollary \ref{CNuessMu} applies to $A$ and we get either $\nuess(A) = \tilde{\mu}(A)$ or $\nuess(A^*) = \tilde{\mu}(A^*)$ (both if $\ind A = 0$). Since $\nuess(A) \leq \tilde{\mu}(A)$ and $\nuess(A^*) \leq \tilde{\mu}(A^*)$ by Proposition \ref{PNuess} and $\tilde{\mu}(A) = \tilde{\mu}(A^*) = \|(A+\pc{\Xb})^{-1}\|^{-1}$ by Lemma \ref{LPLowNorm1}, we conclude
\[\max\{\nuess(A),\nuess(A^*)\}=\mu(A)=\|(A+\pc{\Xb})^{-1}\|^{-1}>0.\]
The second assertion follows immediately from the considerations above.
\end{proof}

\begin{rem}
Starting with a non-invertible operator $B\in\lb{X}$ that is bounded below, one can define the diagonal operator $A:=\diag(\ldots,B,B,B,\ldots)$ on $\Xb$ which is bounded below, but not Fredholm or $\Pc$-Fredholm. Hence, $\nuess(A)$ is positive, but $\mu(A)=0$. Thus some kind of Fredholm condition is necessary. This is why we have
$\speess(A) = \left\{\lambda\in\Cb : \mu(A-\lambda I)<\eps\right\}$ but cannot just write
\[\speess(A) \stackrel{!}{=} \left\{\lambda\in\Cb : \max\left\{\nuess(A-\lambda I),\nuess((A-\lambda I)^*)\right\}<\eps\right\}.\]
However, we will find solutions to this problem in the next sections. Also note that this can not happen if $A \in \Ac(l^p(\Zb,X))$, $\dim X < \infty$ (see Proposition \ref{NuessFinDimBDO} below).
\end{rem}

\subsection{3rd approach: Symmetrization of the problem}
In the two previous approaches we used to look at characteristics of both $A$ and $A^*$
in order to get a complete (symmetric) picture. Now we turn the table in a sense,
firstly symmetrize the operator and secondly determine its essential lower norm.

Given a Banach space $\Yb$ with a uniform approximate projection $\Pc$, we write
$\Yb\oplus\Yb^*$ for the Banach space of all pairs $(x,f)\in\Yb\times\Yb^*$, equipped
with the norm $\|(x,f)\|:=\max\{\|x\|,\|f\|\}$. For $A\in\lb{\Yb}$, $B\in\lb{\Yb^*}$,
write $A\oplus B$ for the operator $(x,f)\mapsto(Ax,Bf)$ in $\lb{\Yb\oplus\Yb^*}$.
The following properties of $A \oplus B$ are easy to check:
\begin{align*}
	& & \left\|A \oplus B\right\| & =  \max\left\{\left\|A\right\|,\left\|B\right\|\right\}, &
	\nu(A \oplus B) & =  \min\left\{\nu(A),\nu(B)\right\}. & &
\end{align*}
To get a similar equality for the essential norm, we have to work a bit more.
Note that $\Pc\oplus\Pc^*=(P_n\oplus P_n^*)_n$ is again a uniform approximate
projection on $\Yb\oplus\Yb^*$.

\begin{prop} \label{Poplusessnorm}
Let $A \oplus B \in \Lc(\Yb\oplus\Yb^*,\Pc\oplus\Pc^*)$. Then
\begin{equation}\label{EOplus}
\left\|A \oplus B + \Kc(\Yb\oplus\Yb^*,\Pc\oplus\Pc^*)\right\|
= \max\left\{\left\|A + \Kc(\Yb,\Pc)\right\|,\left\|B + \Kc(\Yb^*,\Pc^*)\right\|\right\}.
\end{equation}
\end{prop}
\begin{proof}
By the definition, the left hand side
$\left\|A \oplus B + \Kc(\Yb\oplus\Yb^*,\Pc\oplus\Pc^*)\right\|$ of \eqref{EOplus} is
\[
\inf\limits_{K\in\Kc(\Yb\oplus\Yb^*,\Pc\oplus\Pc^*)} \left\|A \oplus B + K\right\|
\leq \inf\limits_{\substack{L\in\Kc(\Yb,\Pc) \\ M\in\Kc(\Yb^*,\Pc^*)}} \left\|A \oplus B + L \oplus M\right\|
\]
where the latter equals
\begin{align*}
\inf\limits_{\substack{L\in\Kc(\Yb,\Pc) \\ M\in\Kc(\Yb^*,\Pc^*)}} \left\|(A+L) \oplus (B+M)\right\|
&= \inf\limits_{\substack{L\in\Kc(\Yb,\Pc) \\ M\in\Kc(\Yb^*,\Pc^*)}} \max\left\{\left\|A+L\right\|,\left\|B+M\right\|\right\}\\
&= \max\left\{\inf\limits_{L\in\Kc(\Yb,\Pc)} \left\|A+L\right\|,\inf\limits_{M\in\Kc(\Yb^*,\Pc^*)} \left\|B+M\right\|\right\}
\end{align*}
which is the right hand side of \eqref{EOplus}, hence proves one direction.

Let $P_{(1)}:\Yb\oplus\Yb^* \to \Yb\oplus\{0\}, (x,f) \mapsto (x,0)$ and
$P_{(2)}:\Yb\oplus\Yb^* \to \{0\}\oplus\Yb^*, (x,f) \mapsto (0,f)$ be the
canonical projections.
Then $\|P_{(1)}\|=\|P_{(2)}\|=1$ and for all $K\in\Kc(\Yb\oplus\Yb^*,\Pc\oplus\Pc^*)$
we have
\begin{align*}
\left\|A \oplus B + K\right\| &\geq \max\left\{\left\|P_{(1)}(A \oplus B + K)P_{(1)}\right\|,\left\|P_{(2)}(A \oplus B + K)P_{(2)}\right\|\right\}\\
&= \max\left\{\left\|A\oplus0 + P_{(1)} K P_{(1)}\right\|,\left\|0\oplus B + P_{(2)} K P_{(2)}\right\|\right\}\\
&\geq \max\left\{\left\|A + \Kc(\Yb,\Pc)\right\|,\left\|B + \Kc(\Yb^*,\Pc^*)\right\|\right\}.
\end{align*}
Taking the infimum over all $K$, we get the reversed inequality.
\end{proof}

\begin{rem}
Note that the naive guess $\nuess(A \oplus B) = \min\left\{\nuess(A),\nuess(B)\right\}$ is wrong in general.
For example, let $\Xb=l^2(\Zb^N,\Cb)$, in which case $\pc{\Xb}=\lc{\Xb}$, and let $A$ be Fredholm on $\Xb$
with index $1$. Then $A \oplus A^*$ has index $0$ and therefore there exists a compact operator $K$
on $\Xb\oplus\Xb$ such that $(A \oplus A^*) + K$ is invertible and in particular bounded below.
Therefore $\nuess(A \oplus A^*)\ge\nu((A \oplus A^*) + K)>0$. However, $A+L$ has index $1$ and
therefore a nontrivial kernel for all $L\in\lc{\Xb}$, so that $\nuess(A)=0$.
\end{rem}

\begin{cor}\label{CNuessOplusMu}
Let $A\in\pb{\Yb}$. \\ Then $\tilde{\mu}(A\oplus A^*)=\mu(A)$ and
either $\nuess(A\oplus A^*)=\mu(A)>0$ or $\nuess(A\oplus A^*)=0$.
\end{cor}
\begin{proof}
We have
\[\tilde{\mu}(A\oplus A^*)
=\lim_{m\to\infty}\nu(A\oplus A^*|_{\im (Q_m\oplus Q_m^*)})
=\lim_{m\to\infty}\min\{\nu(A|_{\im Q_m}),\nu(A^*|_{\im Q_m^*})\} = \mu(A).\]
Corollary \ref{CNuess} applied to $A\oplus A^*$ yields the claim on $\nuess(A\oplus A^*)$.
\end{proof}

Now, we end up with the third characterization of $\|(A+\pc{\Yb})^{-1}\|^{-1}$:
\begin{thm}\label{TNuessOplusY}
Let $A\in\pb{\Yb}$ be $\Pc$-Fredholm and $A+\pc{\Yb}$ contain a Fredholm operator.
Then
\[\nuess(A\oplus A^*)=\tilde{\mu}(A\oplus A^*)=\mu(A) = \|(A+\pc{\Xb})^{-1}\|^{-1}.\]
\end{thm}
Notice that in the case $\Yb = \Xb$ these equalities can be complemented (cf.
Corollary \ref{CFredLX}) by $\max\{\nuess(A),\nuess(A^*)\}=\mu(A)$.

\begin{proof}
W.l.o.g. $A$ is already Fredholm. By \cite[Corollary 1.9]{SeSi3} we have $\Pc$-compact
projections $P,P'$ onto $\ker A$
and parallel to the range of $A$, resp. Then $(P')^*, P^*$ are $\Pc^*$-compact
projections onto $\ker A^*$ and parallel to the range of $A^*$, resp., hence
$R:=P\oplus(P')^*$ is a projection onto the kernel of $A\oplus A^*$ whereas
$R':=P'\oplus P^*$ is a projection parallel to its range. Since both projections
are of the same finite rank there exists an isomorphism $C:\im R\to\im R'$.
Then $A\oplus A^*+R'CR=(I-R')(A\oplus A^*)(I-R)+R'CR$ is invertible, where
$R'CR\in\lc{\Yb\oplus\Yb^*, \Pc\oplus\Pc^*}$ since
\[\lim_{m\to\infty}\|R(Q_m\oplus Q_m^*)\|=\lim_{m\to\infty}\|(Q_m\oplus Q_m^*)R'\|=0.\]
Corollaries \ref{CNuessMu} and \ref{CNuessOplusMu} yield $\nuess(A\oplus A^*) = \tilde{\mu}(A\oplus A^*) = \mu(A)$.
\end{proof}

\subsection{The Hilbert space case}\label{SHilbert}
On a Hilbert space $\Yb$ we consider a sequence of nested orthogonal projections $\Pc=(P_n)_{n\in\Nb}$, i.e. $P_n =P_n^*=P_n^2=P_nP_{n+1}=P_{n+1}P_n$ for all $n\in\Nb$. With this condition $\Pc$ satisfies $(\Pc 1)$ and $(\Pc 2)$. If additionally $(\Pc 3)$ is satisfied, we call $\Pc$ a \emph{H}ermitian \emph{app}roximate \emph{i}dentity (in short: \emph{happi}) and the pairing $(\Yb,\Pc)$ a happi space. In this more particular case of $\Yb$ being a Hilbert space and under this natural assumption on $\Pc$, we will find that now our above results already apply to \emph{all} operators $A\in\pb{\Yb}$.

\begin{lem}\label{TMPI}
An operator $A\in\pb{\Yb}$ is Moore-Penrose invertible (as an element of the $C^*$-algebra
$\lb{\Yb}$) if and only if $\im A$ is closed. In that case $A^+\in\pb{\Yb}$.
\end{lem}
\begin{proof}
The first part is \cite[Theorem 2.4]{Standard}. For the second just
notice that $\pb{\Yb}$ is a $C^*$-subalgebra of $\lb{\Yb}$ and apply
\cite[Corollary 2.18]{Standard}.
\end{proof}

\begin{thm}\label{TPLowNormHilb}
Let $A\in\pb{\Yb}$ on a happi space $(\Yb,\Pc)$. Then
\begin{equation}
\mu(A)=\|(A+\pc{\Yb})^{-1}\|^{-1}.
\end{equation}
Moreover, if $\mu(A) > 0$, then $\tilde{\mu}(A) = \tilde{\mu}(A^*)$.
\end{thm}

\begin{proof}
By Lemma \ref{LPLowNorm1}, $\mu(A)\geq\|(A+\pc{\Yb})^{-1}\|^{-1}$, and
the first assertion obviously holds if $\mu(A)=0$.

So, let $\mu(A)>0$. Then $\nu(A|_{\im Q_m})>0$ for a sufficiently large $m$, hence
$\im(AQ_m)$ is closed (by Lemma \ref{lem:bddbelow}), and we get from Lemma
\ref{TMPI} that $AQ_m$ is Moore-Penrose
invertible with $(AQ_m)^+\in\pb{\Yb}$. Moreover, the compression
$(AQ_m)^+:\im AQ_m \to\im Q_m$ is the inverse of $AQ_m:\im Q_m\to\im AQ_m$,
i.e. $(AQ_m)^+AQ_m=Q_m$ which yields that $(AQ_m)^++\pc{\Yb}$ is a left inverse
for the coset $AQ_m+\pc{\Yb}=A+\pc{\Yb}$.
By the same means we get that $(A^*Q_m)^++\pc{\Yb}$ is a left inverse
for $A^*+\pc{\Yb}$, thus $A+\pc{\Yb}$ is also right invertible. This proves that $A$
is $\Pc$-Fredholm. Applying Lemma \ref{LPLowNorm1}, we get
\[\mu(A)=\tilde{\mu}(A)=\tilde{\mu}(A^*)=\|(A+\pc{\Yb})^{-1}\|^{-1}.\]
\end{proof}

\begin{prop}\label{onesidedHilb}
Let $A\in\Lc(\Yb,\Pc)$ be $\Pc$-Fredholm on a happi space $(\Yb,\Pc)$. Then there is a $K\in\Kc(\Yb,\Pc)$ such that $A+K$ has a one-sided inverse in $\Lc(\Yb,\Pc)$.
\end{prop}

\begin{proof}
By Theorem \ref{TPLowNormHilb} we have that $\mu(A) > 0$ and therefore $\im(AQ_m)$ is closed for $m$ large enough. In order to simplify notations, we may assume that $\im(A)$ is closed. Let $A_0: \ker(A)^\perp \to \im(A)$ be defined by $A_0x = Ax$ for all $x\in\ker(A)^\perp$. Then $A_0$ is invertible by Banach's isomorphism theorem. Now choose orthonormal bases $\left\{\beta_i\right\}_{i\in I}$ and $\left\{\gamma_j\right\}_{j\in J}$ of $\ker(A)$ and $\im(A)^\perp$ respectively. Depending on the cardinalities $\left|I\right|$ and $\left|J\right|$ there is an injection $\iota:I \to J$ or $\iota:J \to I$ (if $\left|I\right| = \left|J\right|$, there is even a bijection). Let us assume that $\left|I\right| \leq \left|J\right|$. Then $\iota$ induces an isometry $\Phi: \ker(A) \to \im(A)^\perp$ by $\Phi(\beta_i) = \gamma_{\iota(i)}$ for all $i\in I$. Let $R_1$ and $R_2$ be orthogonal projections onto $\ker(A)$ and $\im(A)^\perp$ respectively. Then $A + R_2\Phi R_1 = A_0 \oplus \Phi$ is left invertible. More precisely, the Moore-Penrose inverse of $A + R_2\Phi R_1$, which is contained in $\Lc(\Yb,\Pc)$ by Lemma \ref{TMPI}, is a left inverse in this case. It remains to show that $R_2\Phi R_1$ is $\Pc$-compact. Similarly as in the proof of Theorem \ref{TPLowNormHilb} we have that $R_1$ and $R_2$ are $\Pc$-compact since $A$ is $\Pc$-Fredholm. This implies
\[
\left\|R_2\Phi R_1 Q_m\right\| \leq \left\|R_1 Q_m\right\| \to 0
\quad\text{and}\quad
\left\|Q_m R_2\Phi R_1\right\| \leq \left\|Q_m R_2\right\| \to 0
\]
as $m\to\infty$.
\end{proof}

Combining Proposition \ref{PNuess}, Theorem \ref{TPLowNormHilb} and Proposition \ref{onesidedHilb}, this immediately yields

\begin{cor}\label{CnuessHilb}
Let $A\in\Lc(\Yb,\Pc)$ be $\Pc$-Fredholm on a happi space $(\Yb,\Pc)$. Then
\begin{equation}\label{EHilbNuess}
\max\{\nuess(A),\nuess(A^*)\}=\mu(A)=\|(A+\pc{\Yb})^{-1}\|^{-1}.
\end{equation}
More precisely:
\begin{itemize}
	\item If $A+\pc{\Xb}$ contains a left invertible operator, then
	\[\nuess(A) = \mu(A) = \|(A+\pc{\Xb})^{-1}\|^{-1}.\]
	Otherwise, $\nuess(A) = 0$.
	\item If $A+\pc{\Xb}$ contains a right invertible operator, then
	\[\nuess(A^*) = \mu(A) = \|(A+\pc{\Xb})^{-1}\|^{-1}.\]
	Otherwise, $\nuess(A^*) = 0$.
\end{itemize}
\end{cor}

\pagebreak[3]

\begin{cor}\label{CnuessOplusHilb}
Let $A\in\Lc(\Yb,\Pc)$ on a happi space $(\Yb,\Pc)$. Then
\begin{equation}\label{EHilbNuessOplus}
\nuess(A \oplus A^*)=\mu(A)=\|(A+\pc{\Yb})^{-1}\|^{-1}.
\end{equation}
\end{cor}
\begin{proof}
In case of $A$ being $\Pc$-Fredholm apply Corollary \ref{CnuessHilb} to $A \oplus A^*$ and take the observations $\tilde{\mu}(A \oplus A^*) = \tilde{\mu}(A^* \oplus A) = \mu(A \oplus A^*) = \mu(A)$ and $\nuess(A \oplus A^*)=\nuess(A^* \oplus A)$ into account. Then \eqref{EHilbNuess} gives \eqref{EHilbNuessOplus}.

Thus it remains to consider $\nuess(A \oplus A^*)>0$ and to show that $A$ is $\Pc$-Fredholm in  the case. Combining Corollary \ref{CNuess} and Theorem \ref{TPLowNormHilb}, we get that $A \oplus A^*$ is $\Pc$-Fredholm (w.r.t.~$\Pc\oplus\Pc$ in $\Yb\oplus\Yb$). Restricting a $\Pc$-regularizer for $A \oplus A^*$ to the first component yields a $\Pc$-regularizer for $A$ and thus $A$ is $\Pc$-Fredholm.
\end{proof}

\subparagraph{4th approach: Composition of $A$ and $A^*$}
In our three previous approaches to characterize $\|(A+\pc{\Yb})^{-1}\|$, we always combined information
on $A$ and $A^*$. Similar to the formula $\|A^{-1}\|^{-1}=\min(\nu(A),\nu(A^*))$ from
Lemma \ref{lem:bddbelow}, one always needs to look at both of them. What we did so far is
to consider the following ideas:
\begin{itemize}
\itemsep-1mm
\item Take the numbers $\tilde\mu(A)$ and $\tilde\mu(A^*)$ and look at their minimum $\mu(A)$.
\item Take the numbers $\nuess(A)$ and $\nuess(A^*)$ and look at their maximum.
\item Combine both operators to $A\oplus A^*$ and look at the number $\nuess(A\oplus A^*)$.
\end{itemize}
These expressions are found, most notably, in Theorems \ref{TPLowNorm},
\ref{TPLowNormHilb}; in Corollary \ref{CFredLX} and Corollary \ref{CnuessHilb};
as well as in Theorem \ref{TNuessOplusY} and Corollary \ref{CnuessOplusHilb}.
A different approach to the same goal,
to find $\|(A+\pc{\Yb})^{-1}\|$, could be to couple the operators $A$ and $A^*$ to
a new operator via composition.

\begin{cor}
Let $A\in\Lc(\Yb,\Pc)$ on a happi space $(\Yb,\Pc)$. Then
\[\min\left\{\sqrt{\nuess(AA^*)},\sqrt{\nuess(A^*A)}\right\}=\mu(A)=\|(A+\pc{\Yb})^{-1}\|^{-1}.\]
\end{cor}
\begin{proof}
If $A$ is $\Pc$-Fredholm then, since $\pb{\Yb}/\pc{\Yb}$ is a $C^*$-algebra,
\[\|(AA^*+\pc{\Yb})^{-1}\|^{-1} = \|(A+\pc{\Yb})^{-1}\|^{-2} = \|(A^*A+\pc{\Yb})^{-1}\|^{-1}\]
and the assertion follows from Corollary \ref{CnuessHilb} applied to the self-adjoint operators $AA^*$ and $A^*A$. So it remains to show that $\min\left\{\sqrt{\nuess(AA^*)},\sqrt{\nuess(A^*A)}\right\} > 0$ implies that $A$ is $\Pc$-Fredholm. Combining  Proposition \ref{PNuess} and Theorem \ref{TPLowNormHilb}, we get that $AA^*$ and $A^*A$ are both $\Pc$-Fredholm.  Consequently, $A+\pc{\Yb}$ has both right and left inverses. Thus $A$ is also $\Pc$-Fredholm.
\end{proof}

\subsection{The case of finite-dimensional entries}\label{SFinDim}

Let us now consider the case $\Xb=l^p(\Zb^N,X)$, with $\dim X<\infty$, which we already addressed in Remark \ref{RFinDim} and Proposition \ref{PFredh}. Then $\pc{\Xb}\subset\lc{\Xb}$ holds (since every $\Pc$-compact operator $K$ is the norm limit of the sequence of finite rank operators $KP_n$), hence every $\Pc$-Fredholm operator is Fredholm.
Actually, the $\Pc$-Fredholm property coincides with Fredholmness by Proposition \ref{PFredh}, and we even have

\begin{prop}\label{PEssNormFinDim}
Let $\dim X<\infty$ and $A\in\Lc(\Xb,\Pc)$. Then
\begin{align*}
\|A+\lc{\Xb}\|=\|A+\pc{\Xb}\|&=\|A^*+\lc{\Xb^*}\|=\|A^*+\lc{\Xb^*,\Pc^*}\|,\\
\|(A+\lc{\Xb})^{-1}\|^{-1}=\|(A+\pc{\Xb})^{-1}\|^{-1}&=\|(A^*+\lc{\Xb^*})^{-1}\|^{-1}=\|(A^*+\lc{\Xb^*,\Pc^*})^{-1}\|^{-1}.
\end{align*}
\end{prop}
Note that the essential and $\Pc$-essential norm obviously do not coincide if
$\dim X =\infty$, just consider the operators $P_1$ and $I-P_1$.
\begin{proof}
Since ($A$ Fredholm $\Rightarrow$ $A$ $\Pc$-Fredholm (by Proposition \ref{PFredh})
$\Rightarrow$ $A^*$ $\Pc^*$-Fredholm $\Rightarrow$ $A^*$ Fredholm (by
$\lc{\Xb^*}\supset\lc{\Xb^*,\Pc^*}$) $\Rightarrow$ $A$ Fredholm),
all terms in the second line are simultaneously zero or non-zero. If they are
non-zero, then Proposition \ref{PFredh} provides a generalized inverse $B\in \pb{\Xb}$
for $A$, and the second asserted line follows from the first one applied to $B$.

For the first line we recall Proposition \ref{PEssNorm} which shows that the
$\Pc$-essential norm is invariant under passing to the adjoint $A^*$.
In the cases $p \in \{0\}\cup(1,\infty)$, where $\pc{\Xb}=\lc{\Xb}$ holds,
$\|A+\lc{\Xb}\|=\|A+\pc{\Xb}\|$ is obvious as well, and we next prove
this equality for the cases $p = 1$ and $p = \infty$:

$p=1$: Let $\eps>0$ and choose $K\in\Kc(\Xb)$ such that
$\|A+K\|\leq\|A+\Kc(\Xb)\|+\eps$ and $m_0 \in \Nb$ such that
$\|Q_mK\| \leq \eps$ for all $m \geq m_0$, which is possible because
$Q_m$ converges strongly to $0$ as $m \to \infty$ and $K$ is compact.
Now we can proceed as in Proposition \ref{PEssNorm}:
\[\|Q_mA\|=\|A-P_mA\| \geq \|A+\Kc(\Xb)\| \geq \|A+K\|-\eps
	\geq \|Q_m(A+K)\|-\eps \geq \|Q_mA\|-2\eps\]
for all $m \geq m_0$ and therefore
$\|A+\Kc(\Xb)\| = \lim_{m \to \infty} \|Q_mA\| = \|A+\Kc(\Xb,\Pc)\|$.

$p=\infty$: Let $\eps>0$ and choose $K\in\Kc(\Xb)$ such that
$\|A+K\|\leq\|A+\Kc(\Xb)\|+\eps$ and $m_0 \in \Nb$ such that
$\|KQ_m|_{\Xb_0}\| \leq \eps$ for all $m \geq m_0$, which is possible
because $(Q_m|_{\Xb_0})^* = Q_m|_{\ell^1(\Zb^N,X^*)}$ converges strongly
to $0$ as $m \to \infty$ and $(K|_{\Xb_0})^*$ is compact.
Now we can proceed as before, using Lemma \ref{LInverses}~a):
\begin{align*}
\|AQ_m\|&=\|A-AP_m\| \geq \|A+\Kc(\Xb)\| \geq \|A+K\|-\eps \geq \|(A+K)Q_m\|-\eps\\
&\geq \|(A+K)Q_m|_{\Xb_0}\|-\eps \geq \|AQ_m|_{\Xb_0}\|-\|KQ_m|_{\Xb_0}\|-\eps \geq \|AQ_m\|-2\eps
\end{align*}
for all $m \geq m_0$ and therefore
$\|A+\Kc(\Xb)\| = \lim_{m \to \infty} \|AQ_m\| = \|A+\Kc(\Xb,\Pc)\|$.
\newpage 

Up to now we have $\|A+\lc{\Xb}\|=\|A+\pc{\Xb}\|=\|A^*+\lc{\Xb^*,\Pc^*}\|$
for all $p$, and hence the complete first line for all $p<\infty$ by taking
a circuit using the natural and well known dualities:
$\|A+\lc{\Xb}\|=\|A+\pc{\Xb}\|=\|A^*+\lc{\Xb^*,\Pc^*}\|=\|A^*+\lc{\Xb^*}\|.$
It remains to prove that in the case $p=\infty$ the essential norm of $A$
coincides with the essential norm of $A^*$. Actually, such a claim this is not
true in general Banach spaces, as was shown in \cite{Axler}. Anyway, for our
particular case $\Xb=l^\infty(\Zb^N,X)$ with $\dim X <\infty$ one can utilize
a further observation from \cite{Axler}: $\Xb$ is the dual of another Banach space,
namely $\Yb=l^1(\Zb^N,X^*)$, and the adjoint of the canonical embedding
$E:\Yb\to\Yb^{**}$ is an operator
$E^*:\Xb^{**}\to\Xb$ onto $\Xb$ of the norm $1$. For every $K\in\lc{\Xb^{**}}$
(with $J$ denoting the canonical embedding $J:\Xb\to\Xb^{**}$)
\[\|A^{**}+K\|\geq\|(A^{**}+K)J\|\geq\|E^*(A^{**}+K)J\|=\|A+E^*KJ\|\geq\|A+\lc{\Xb}\|.\]
Taking the infimum over all $K$ it follows
$\|A^{**}+\lc{\Xb^{**}}\|\geq\|A+\lc{\Xb}\|$.
Since the adjoint of any compact operator is compact the desired equality follows by
\[\|A+\lc{\Xb}\|\geq\|A^*+\lc{\Xb^*}\|\geq\|A^{**}+\lc{\Xb^{**}}\|\geq\|A+\lc{\Xb}\|.\]
\end{proof}

In analogy to $\nuess(A)$ we denote the classical (w.r.t.~$\lc{\Xb}$) essential
lower norm by $\nuessc(A)$:
\[\nuessc(A):=\sup\{\nu(A+K):K\in\lc{\Xb}\}\]
and we obtain the following improvement and completion of the results in the previous
sections:

\begin{thm}\label{TEssLowerNormFinDim}
Let $\dim X<\infty$ and $A\in\Lc(\Xb,\Pc)$. Then
\begin{equation}\label{E1}
\nuessc(A\oplus A^*)=\nuess(A\oplus A^*)
=\mu(A)=\|(A+\pc{\Xb})^{-1}\|^{-1}=\|(A+\lc{\Xb})^{-1}\|^{-1}.
\end{equation}
Moreover, if $A$ is Fredholm of index zero, then
\[\nuessc(A)=\nuessc(A^*)=\nuess(A)=\nuess(A^*)=\mu(A)=\tilde{\mu}(A)=\tilde{\mu}(A^*) > 0.\]
Conversely, if $\nuessc(A) > 0$ and $\nuessc(A^*) > 0$, then $A$ is Fredholm of index zero.
\end{thm}

\begin{proof}
Let $B$ be a Fredholm operator of index zero on a Banach space $\Yb$.
Then,
\begin{align*}
\nuessc(B)&=\sup\{\nu(B+K):K\in\lc{\Yb}, B+K \text{ bounded below}\} &\notag\\
&=\sup\{\nu(B+K):K\in\lc{\Yb}, B+K \text{ invertible}\} &\text{($\ind(B+K)=0)$}\\
&=\sup\{\|(B+K)^{-1}\|^{-1}:K\in\lc{\Yb}, B+K \text{ invertible}\} &\text{(Lemma \ref{lem:bddbelow})}\notag\\
&=(\inf\{\|(B+K)^{-1}\|:K\in\lc{\Yb}, B+K \text{ invertible}\})^{-1}&\notag\\
&=\|(B+\lc{\Yb})^{-1}\|^{-1}.\notag&
\end{align*}
Hence, if $A$ is Fredholm of index zero, then
$\nuessc(A^*)=\nuessc(A)=\|(A+\lc{\Xb})^{-1}\|^{-1}$. Together with the obvious
estimates $\nuessc(A)\geq\nuess(A)$ and $\nuessc(A^*)\geq\nuess(A^*)$,
Corollary \ref{CFredLX} and Proposition \ref{PEssNormFinDim} yield
$\nuessc(A)=\nuessc(A^*)=\nuess(A)=\nuess(A^*)=\mu(A)=\|(A+\lc{\Xb})^{-1}\|^{-1}.$
Corollary \ref{CNuess} completes this equality: $\nuess(A)=\tilde{\mu}(A)$, $\nuess(A^*)=\tilde{\mu}(A^*)$.

Whenever $A$ is a (general) Fredholm operator then the operator $A\oplus A^*$ is
Fredholm of index zero, and by the above it follows
that $\nuessc(A\oplus A^*)=\|(A\oplus A^*+\lc{\Xb\oplus\Xb^*})^{-1}\|^{-1}$.
If $B$ is a regularizer for $A$ then the latter is
$\|B\oplus B^*+\lc{\Xb\oplus\Xb^*}\|^{-1}$. This equals $\|B+\lc{\Xb}\|^{-1}$
which is proved in the same way as Proposition \ref{Poplusessnorm}, taking
Proposition \ref{PEssNormFinDim} into account. Thus
$\nuessc(A\oplus A^*)=\|(A+\lc{\Xb})^{-1}\|^{-1}$. Since
\begin{align*}
\nuessc(A\oplus A^*)\geq\nuess(A\oplus A^*)=\mu(A)\geq \|(A+\pc{\Xb})^{-1}\|^{-1}=\|(A+\lc{\Xb})^{-1}\|^{-1},\label{EFinDimEst}
\end{align*}
(cf. Theorem \ref{TNuessOplusY}, Lemma \ref{LPLowNorm1}, Proposition \ref{PEssNormFinDim})
the claim \eqref{E1} easily follows.

If $A$ is not Fredholm then $A$ is not normally solvable or $A$ has infinite
dimensional kernel or $A^*$ has infinite dimensional kernel. In the first and second
case this remains true for all $A|_{\im Q_m}$, resp., hence all $\nu(A|_{\im Q_m})$
equal zero. In the latter case all $\nu(A^*|_{\im Q^*_m})$
must be zero, and we conclude that $\mu(A)=0$.
Moreover, $A\oplus A^*$ is not normally solvable or has infinite dimensional kernel,
hence $0=\nuessc(A\oplus A^*)\geq\nuess(A\oplus A^*)$.
Thus \eqref{E1} also holds in this case.

If $\nuessc(A) > 0$, then there exists a $K\in\lc{\Xb}$ such that $\nu(A+K) > 0$. This implies that $A+K$ is injective and normally solvable by Lemma \ref{lem:bddbelow}. This implies that $A$ is normally solvable and has finite-dimensional kernel. Similarly, if $\nuessc(A^*) > 0$, then $A^*$ is normally solvable and has finite-dimensional kernel. In particular, $A$ is Fredholm if both $\nuessc(A) > 0$ and $\nuessc(A^*) > 0$ hold. Moreover, $\nu(A+K) > 0$ implies $\ind(A) = \ind(A+K) \leq 0$ whereas $\nu(A^*+L) > 0$ implies $-\ind(A) = \ind(A^*) = \ind(A^*+L) \leq 0$. Thus the index of $A$ has to be zero. This proves the last part.
\end{proof}

Actually, there is an even more abstract version of Theorem \ref{TEssLowerNormFinDim}
within the $\Pc$-framework:
\begin{thm}\label{TEssLowerNormFinDimGen}
Let $\Yb$ be a Banach space with a uniform approximate identity $\Pc=(P_n)$
consisting of finite rank projections $P_n$. Then for every $A\in\pb{\Yb}$ which
has the $\Pc$-dichotomy
\begin{equation}\label{E2}
\nuess(A\oplus A^*)=\mu(A)=\|(A+\pc{\Yb})^{-1}\|^{-1}.
\end{equation}
\end{thm}
\begin{proof}
Let $A$ be $\Pc$-Fredholm. Then $A\oplus A^*$ is $\Pc\oplus\Pc^*$-Fredholm,
Fredholm of index $0$, and
\begin{align*}
\nuess(A\oplus A^*)&=\sup\{\nu(A\oplus A^*+K):K\in\lc{\Yb\oplus \Yb^*,\Pc\oplus\Pc^*}, A\oplus A^*+K \text{ bounded below}\} \\
&= \ldots=\|(A\oplus A^*+\lc{\Yb\oplus\Yb^*,\Pc\oplus\Pc^*})^{-1}\|^{-1}
\end{align*}
as above. It follows that
$\nuess(A\oplus A^*)=\|(A\oplus A^*+\lc{\Yb\oplus\Yb^*,\Pc\oplus\Pc^*})^{-1}\|^{-1}=\|(A+\pc{\Yb})^{-1}\|^{-1}$
also by Proposition \ref{Poplusessnorm} and Proposition \ref{PEssNorm},
taking Theorem \ref{TPFredh} for a regularizer $B$ of $A$ into account.
Theorem \ref{TNuessOplusY} completes \eqref{E2}.

If $A$ is not $\Pc$-Fredholm then $A$ is $\Pc$-deficient. Thus
$A\oplus A^*$ is $\Pc\oplus\Pc^*$-deficient from both sides,
and it easily follows $\nuess(A\oplus A^*)=\mu(A)=0$.
Thus \eqref{E2} also holds in this case.
\end{proof}

In the particular case $N=1$, we also have the following Proposition for band-dominated operators:

\begin{prop} \label{NuessFinDimBDO}
Let $\dim X<\infty$ and $A\in\Ac(l^p(\Zb,X))$. Then
\[\max\{\nuess(A),\nuess(A^*)\}=\mu(A)=\|(A+\pc{\Xb})^{-1}\|^{-1}=\|(A+\lc{\Xb})^{-1}\|^{-1}.\]
\end{prop}

\begin{proof}
The equality of $\mu(A)$, $\|(A+\pc{\Xb})^{-1}\|^{-1}$ and $\|(A+\lc{\Xb})^{-1}\|^{-1}$ holds by Theorem \ref{TEssLowerNormFinDim}.
If $A$ is Fredholm, then Corollary \ref{CFredLX} implies the remaining equality. If $A$ is not Fredholm, then $A$ is not even semi-Fredholm by \cite[Theorem 4.3]{Se:SemiFred}, i.e. either $A$ is not normally solvable or both $A$ and $A^*$ have an infinite-dimensional kernel. This also remains true for all $B \in A+\lc{\Xb}$. It follows $\nuess(A) \leq \nuessc(A) = 0$ and $\nuess(A^*) \leq \nuessc(A^*) = 0$ by the definition of $\nuessc$ and Lemma \ref{lem:bddbelow}, hence $\max\{\nuess(A),\nuess(A^*)\} = 0$.
\end{proof}

\subparagraph{5th approach: Approximation numbers and singular values}
In this more comfortable situation $\dim X<\infty$ we can study further characterizations
of $\mu(A)$.
\begin{defn} (cf. \cite{Zeman,SeDis,SeSi3})
For an operator $A\in\lb{\Xb}$ we define the $m$th lower Bernstein numbers
and (one-sided) approximation numbers by
\begin{align*}
	B_m(A) & := \sup\{\nu(A|_V): \dim \Xb/V < m\},\\
	s_m^r(A)& := \inf \{\|A-F\|: F\in\lb{\Xb}, \dim \ker F \geq m\},\\
	s_m^l(A)& := \inf \{\|A-F\|: F\in\lb{\Xb}, \dim \coker F  \geq m\}.
\end{align*}
Moreover we introduce the limits
\begin{align*}
	B(A) & := \lim_{m\to\infty}\min\{B_m(A), B_m(A^*)\},\\
	S(A) & := \lim_{m\to\infty}\min\{s_m^r(A), s_m^l(A)\},
\end{align*}
whose existence is proved by monotonicity as in Lemma \ref{LMuLim}.
\end{defn}

\begin{thm}\label{TAppNr}
Let $\dim X<\infty$ and $A\in\pb{\Xb}$. Then
\begin{equation}
\mu(A)=B(A)=S(A).
\end{equation}
\end{thm}

\begin{proof}
By definition $\nu(A|_{\im Q_m})\leq B_{\dim\ker Q_m+1}(A)$, and $\mu(A)\leq B(A)$ easily
follows. Furthermore $B(A)\leq S(A)$ holds by \cite[Proposition 2.9]{SeSi3} or
\cite[Proposition 1.36]{SeDis}.
Also, it is easily seen that all these numbers are zero if $A$ is not Fredholm.

Let $A$ be Fredholm.
Assume that there are constants $d,e$ such that $\tilde{\mu}(A)<d<e<S(A)$.
Then $\nu(A|_{\im Q_n})<d$ for all
 $n\in\Nb$. This means that there exist $y_n\in\im Q_n$ such that $\|Ay_n\|<d\|y_n\|$,
respectively. Recalling the sequence $(F_n)$ from Proposition \ref{PFn}
we further conclude that $\|AF_ly_n\|<d\|F_ly_n\|$ for sufficiently
large $l$, since $\|[A,F_l]\|$ tends to zero and $\|F_ly_n\|$ tends to $\|y_n\|$ as
$l\to\infty$. Fix such an $l$ (which depends on $n$) such that also
$F_lP_n=P_nF_l=P_n$ holds, and define $z_n:=\|F_ly_n\|^{-1}F_ly_n$,
respectively. Then $z_n\in\im Q_n$ is still true since $F_ly_n=F_lQ_ny_n=Q_nF_ly_n$.

Next, we fix $m\in\Nb$ and choose numbers $n_1,\ldots,n_m$ as follows:
Set $n_1:=1$. Given $n_i$ choose $l_i$ such that $P_{l_i}F_{n_i}=F_{n_i}P_{l_i}=F_{n_i}$.
Then $z_{n_i}$ is in the range of $P_{l_i}Q_{n_i}$.
Furthermore choose $k_i>l_i$ such that $\|Q_{k_{i}}AP_{l_i}\|<2^{-i-1}(e-d)$ and
$n_{i+1}>k_i$ such that $\|P_{k_i}AQ_{n_{i+1}}\|<2^{-i-2}(e-d)$.
For every $i$ let $R_i$ be a projection of norm $1$ onto $\spn\{z_{n_i}\}$ and such
that $R_i=R_iP_{l_i}Q_{n_i}$, respectively, and define $S_m:=\sum_{i=1}^mR_i$.
Then $S_m$ is a projection of rank $m$, $R_i=R_iS_m$ for all $i=1,\ldots,m$,
and $\|S_m\|=1$. Moreover,
\begin{align*}
\|AS_mx\|&=\left\|\sum_{i=1}^m AR_ix\right\|
=\left\|\sum_{i=1}^m P_{k_i}Q_{k_{i-1}}AR_ix
+\sum_{i=1}^m P_{k_{i-1}}AR_ix+\sum_{i=1}^m Q_{k_i}AR_ix\right\|\\
&\leq\left\|\sum_{i=1}^m P_{k_i}Q_{k_{i-1}}AR_ix\right\|
		+\left\|\sum_{i=1}^m P_{k_{i-1}}AQ_{n_i}R_ix\right\|
		+\left\|\sum_{i=1}^m Q_{k_i}AP_{l_i}R_ix\right\|\\
&\leq\left\|\sum_{i=1}^m P_{k_i}Q_{k_{i-1}}AR_ix\right\|
		+\sum_{i=1}^m2^{-i-1}(e-d)\|x\|+\sum_{i=1}^m2^{-i-1}(e-d)\|x\|.
\end{align*}
For the first term we have
\begin{align*}
\left\|\sum_{i=1}^m P_{k_i}Q_{k_{i-1}}AR_ix\right\|^p
=\sum_{i=1}^m \|P_{k_i}Q_{k_{i-1}}AR_ix\|^p
\leq d^p\sum_{i=1}^m \|R_ix\|^p =d^p\|S_mx\|^p \leq d^p\|x\|^p
\end{align*}
in the cases $p\in[1,\infty)$, and similarly for $p\in\{0,\infty\}$.
Thus $\|AS_mx\|\leq e\|x\|$ for all $x$, and hence
\[s_m^r(A) =\inf\{\|A-F\|:\dim\ker F \geq m\}\leq\|A-A(I-S_m)\|=\|AS_m\|\leq e<S(A).\]
Sending $m\to\infty$ we arrive at a contradiction. Thus $\tilde{\mu}(A)\geq S(A)$.

Since $A$ is $\Pc$-Fredholm by Proposition \ref{PFredh}, we can apply Lemma \ref{LPLowNorm1} to obtain
\[S(A) \geq B(A) \geq \mu(A) = \tilde{\mu}(A) \geq S(A).\]
\end{proof}

\begin{rem}
In the case $\Xb$ being a Hilbert space we even have
\[\mu(A)=
\Sigma(A):=\lim_{m\to\infty}\min\{\sigma_m(A),\sigma_m(A^*)\},\]
where $\sigma_m(A)$ denotes the $m$-th singular value of $A$ (see
\cite[Corollary 2.12]{SeSi3}).
\end{rem}

\subsection{On the characterization of essential (pseudo)spectra}

We have seen that in the following cases there are several characterizations of
the essential lower norm
\begin{itemize}
\item Band-dominated operators on all sequence spaces $\Xb$
\item $\pb{\Xb}$-operators in the case $\dim X<\infty$
\item $\pb{\Yb}$-operators on happi spaces $(\Yb,\Pc)$,
\end{itemize}
namely $\mu(A)$, $\nuess(A\oplus A^*)$ in all these cases and additionally the
essential lower norms of $AA^*$ and $A^*A$ in the happi case. The case
$\dim X<\infty$ offers the largest collection of characterizations, including
also $B(A)$ and $S(A)$, and the classical (non-$\Pc$) essential lower norm.

Each of them permits to give an equivalent definition of $\Pc$-essential spectra
and pseudospectra:

\begin{thm}
\begin{enumerate}
	\item[a)] Let $A \in \Alpr$. Then
	\begin{align*}
	\spess(A) &= \{\lambda\in\Cb:\|(A-\lambda I+\pc{\Xb})^{-1}\|^{-1}=0\} = \{\lambda\in\Cb:\mu(A-\lambda I)=0\},\\
	\speess(A) &= \{\lambda\in\Cb:\|(A-\lambda I+\pc{\Xb})^{-1}\|^{-1}<\eps\} = \{\lambda\in\Cb:\mu(A-\lambda
 I)<\eps\}\\
 	&= \spess(A) \cup \{\lambda\in\Cb:\max\{\nuess(A-\lambda I),\nuess((A-\lambda I)^*)\}<\eps\}.
	\end{align*}
	If even $\Xb=l^p(\Zb,X)$ with $\dim X < \infty$, then
	\[\speess(A) = \{\lambda\in\Cb:\max\{\nuess(A-\lambda I),\nuess((A-\lambda I)^*)\}<\eps\}.\]
	\item[b)] Let $(\Yb,\Pc)$ be a happi space and $A \in \Lc(\Yb,\Pc)$. Then
	\begin{align*}
	\spess(A) &= \{\lambda\in\Cb:\mu(A-\lambda I)=0\}\\
	&= \{\lambda\in\Cb:\nuess((A-\lambda I)\oplus(A-\lambda I)^*)=0\}\\
	&= \{\lambda\in\Cb:\min\{\sqrt{\nuess((A-\lambda I)(A-\lambda I)^*)},\sqrt{\nuess((A-\lambda I)^*(A-\lambda I))}\}=0\},\\
	\speess(A) &= \{\lambda\in\Cb:\mu(A-\lambda I)<\eps\}\\
	&= \{\lambda\in\Cb:\nuess((A-\lambda I)\oplus(A-\lambda I)^*)<\eps\}\\
	&= \{\lambda\in\Cb:\min\{\sqrt{\nuess((A-\lambda I)(A-\lambda I)^*)},\sqrt{\nuess((A-\lambda I)^*(A-\lambda I))}\}<\eps\}\\
	&= \spess(A) \cup \{\lambda\in\Cb:\max\{\nuess(A-\lambda I),\nuess((A-\lambda I)^*)\}<\eps\}.
	\end{align*}
	\item[c)] Let $A \in \Lc(\Xb,\Pc)$ and $\dim X < \infty$. Then
	\begin{align*}
	\spess(A) &= \{\lambda\in\Cb:\|(A-\lambda I+\lc{\Xb})^{-1}\|^{-1}=0\}\\
	&= \{\lambda\in\Cb:\mu(A-\lambda I)=0\}\\	
	&= \{\lambda\in\Cb:\nuess((A-\lambda I)\oplus(A-\lambda I)^*)=0\}\\
	&= \{\lambda\in\Cb:B(A-\lambda I)=0\}\\
	&= \{\lambda\in\Cb:S(A-\lambda I)=0\},\\
	\speess(A) &= \{\lambda\in\Cb:\|(A-\lambda I+\lc{\Xb})^{-1}\|^{-1}<\eps\}\\
	&= \{\lambda\in\Cb:\mu(A-\lambda I)<\eps\}\\
	&= \{\lambda\in\Cb:\nuess((A-\lambda I)\oplus(A-\lambda I)^*)<\eps\}\\
	&= \{\lambda\in\Cb:B(A-\lambda I)<\eps\}\\
	&= \{\lambda\in\Cb:S(A-\lambda I)<\eps\}\\
 	&= \spess(A) \cup \{\lambda\in\Cb:\max\{\nuess(A-\lambda I),\nuess((A-\lambda I)^*)\}<\eps\}.
	\end{align*}
	\item[d)] If the conditions in b) and c) are both fulfilled, we additionally have
	\begin{align*}
	\spess(A) &= \{\lambda\in\Cb:\Sigma(A-\lambda I)=0\},\\
	\speess(A) &= \{\lambda\in\Cb:\Sigma(A-\lambda I)<\eps\},\\
	\end{align*}
\end{enumerate}
\end{thm}

\section{On finite sections}\label{sec:FSM}
In this section we apply our results, in particular Corollary \ref{CEssNorm}, in the context of asymptotic inversion of an operator.
\subparagraph{Stability}
Let $A\in\pb{\Xb}$. For the approximate solution of an equation $Ax=b$ or, likewise,
for the approximation of the pseudospectrum $\spe(A)$, one is looking for approximations
of the inverse of $A$ (or of $A-\lambda I$, respectively) by operators that can be stored and
worked with on a computer.
\newpage 

Assuming invertibility of $A$, a natural idea is to take a sequence of operators $A_1,A_2,...$
in $\pb{\Xb}$ with $A_n\pto A$ as $n\to\infty$, and to hope that, for all sufficiently large $n$,
also $A_n$ is invertible and $A_n^{-1}\pto A^{-1}$.
It turns out (see e.g. \cite[Theorem 6.1.3]{LimOps}, \cite[Corollary 1.77]{Marko},
\cite[Propositions 1.22, 1.29 and Corollary 1.28]{SeDis}) that this hope will be fulfilled
if and only if the sequence $(A_n)$ is {\sl stable}, meaning that there is an $n_0$ such
that all $A_n$ with $n\ge n_0$ are invertible and $\sup_{n\ge n_0}\|A_n^{-1}\|<\infty$.
In short:
\begin{equation}\label{eq:stab}
(A_n) \ \textrm{is stable}\qquad :\iff\qquad \limsup_{n\to\infty}\|A_n^{-1}\|<\infty
\end{equation}
After a positive answer to this qualitative question about stability, one will ask about
quantities:\\[2mm]
\begin{tabular}{lll}
&(Q1) & How large is the $\limsup$ in \eqref{eq:stab}?\\
&(Q2) & Is it possibly a limit?\\
&(Q3) & What is the asymptotics of the condition numbers $\kappa(A_n)=\|A_n\|\|A_n^{-1}\|$?\\
&(Q4) & What is the asymptotic behaviour of the pseudospectra $\spe(A_n)$?
\end{tabular}\\[2mm]
There are different approaches \cite{RochFredTh,RochOnFS,LiFS,SeDis,SeSi4} to deal with these questions. We will discuss one of them and we will focus on questions (Q1) and (Q3). The discussion of (Q2) and (Q4)  is postponed to a further paper, \cite{HagLiSei:FSM}, as it would overstretch both length and scope of the current paper. Moreover, we will restrict ourselves here to studying sequences $(A_n)$ of so-called finite sections (see below) of an $A\in\pb{\Xb}$ as opposed to \cite{HagLiSei:FSM}, where we look at more general elements of an algebra of such sequences.

\subparagraph{The stacked operator}
The idea is to identify the whole sequence $(A_n)$ with one single operator, denoted by $\oplus A_n$, that acts componentwise on a direct sum of infinitely many copies of $\Xb$.
To make this precise, first extend the sequence $(A_n)_{n\in\Nb}$ to the index set $\Zb$,
for example by $(A_n)_{n\in\Zb}:=(\cdots,cI,cI,A_1,A_2,\cdots)$ with some $c\ne 0$, and then,
recalling that $\Xb=l^p(\Zb^N,X)$, put
\[
\Xb' := \oplus^{l^p}_{n\in\Zb}\Xb := l^p(\Zb,\Xb) \cong l^p(\Zb^{N+1},X).
\]
Now each bounded sequence
$(A_n)_{n\in\Zb}\subset\lb{\Xb}$ acts as a diagonal operator on
$\Xb'=l^p(\Zb,\Xb)$. We denote this operator by $\oplus A_n:\Xb'\to\Xb'$ and refer to it
as the {\sl stacked operator} of the sequence $(A_n)$. Then (see \cite[Section 2.4.1]{Marko} or
\cite[Section 6.1.3]{HabilMarko})
\begin{equation} \label{eq:normsup}
\|\oplus A_n\| = \sup_n\|A_n\|.
\end{equation}
In order to avoid confusion we denote the approximate identity on $\Xb'=l^p(\Zb^{N+1},X)$
by $\Pc'=(P_n')$, where $P_n'=\chi_{\{-n,...,n\}^{N+1}}I$ and
$P_n=\chi_{\{-n,...,n\}^{N}}I$.

\subparagraph{Finite sections}
Given $A\in\pb{\Xb}$, a natural construction for the approximating sequence $(A_n)$ is to
look at the so-called {\sl finite sections}
\begin{equation} \label{eq:An}
A_n:=P_nAP_n,\qquad n\in\Nb,
\end{equation}
of $A$. Here $A_n$ is understood as operator $\im P_n\to\im P_n$ and is hence represented by a finite matrix. For completeness, put $P_n:=0$
for $n\in\Zb\setminus\Nb$, so that the same formula \eqref{eq:An} gives $A_n=0$ then.
From $P_n\pto I$ it follows that $A_n\pto A$ as $n\to\infty$, where
we freely identify $A_n$ with its extension by zero to the whole space $\Xb$. However, when writing $A_n^{-1}$, we clearly mean the inverse
(or its extension by zero to $\Xb$) of $A_n:\im P_n\to\im P_n$.
For the study of stability of a sequence it is more convenient to have all invertibility problems on the same space. To this end we fix a $c>0$ and look at the extensions
\begin{equation} \label{eq:Anc}
A_{n,c}:=P_n AP_n + cQ_n,\qquad n\in\Zb,
\end{equation}
of $A_n$, by $c$ times the identity, to $\Xb$. Clearly, also $A_{n,c}\pto A$ as $n\to\infty$. Now $P_n=0$ implies $A_{n,c}=cI$ for $n\in\Zb\setminus\Nb$.
Note that $A_n$ is invertible on $\im P_n$ if and only if $A_{n,c}$ is invertible on $\Xb$, and that
\begin{eqnarray} \label{eq:normAnc}
A_{n,c}^{-1}=A_n^{-1}+c^{-1}Q_n,
\quad\textrm{so that}\quad
\|A_{n,c}^{-1}\|=\max(\|A_n^{-1}\|,c^{-1}),
\end{eqnarray}
whence both sequences, $(A_n)$ and $(A_{n,c})$, are stable\footnote{We also call a bi-infinite operator sequence $(A_n)_{n\in\Zb}$ stable if it is subject to \eqref{eq:stab} (with $\infty$ referring to $+\infty$), i.e.~if its semi-infinite part $(A_n)_{n\in\Nb}$ is stable. The other part, $(A_n)_{n\in\Zb\setminus\Nb}$, as we defined it, is uncritical anyway.} at the same time. Note that the choice $c\ge\|A\|\ge\|A_n\|$ implies $c^{-1}\le\|A_n\|^{-1}\le\|A_n^{-1}\|$, so that $\|A_{n,c}^{-1}\|=\|A_n^{-1}\|$, by \eqref{eq:normAnc}.

\begin{thm} \label{thm:59}
(\cite[Theorem 6.1.6, Lemma 6.1.7]{LimOps} or \cite[Proposition 2.22, Theorem 2.28]{Marko})\\
Let $A\in\pb{\Xb}$, $c>0$ and $(A_n), (A_{n,c})\subset\pb{\Xb}$ be the sequences defined in \eqref{eq:An} and \eqref{eq:Anc}. Then
\begin{itemize}
\itemsep-1mm
\item The stacked operators $\oplus A_n$ and $\oplus A_{n,c}$ are in $\lb{\Xb',\Pc'}$.
\item If $A$ is rich then the stacked operators $\oplus A_n$ and $\oplus A_{n,c}$ are rich.
\item If $A$ is band-dominated then the stacked operators $\oplus A_n$ and $\oplus A_{n,c}$ are band-dominated.
\item $(A_n)$ is stable if and only if the stacked operator $\oplus A_{n,c}$ is $\Pc'$-Fredholm.
\end{itemize}
\end{thm}

Combining Theorems \ref{thm:59} and \ref{TBdORichHS}, we get:

\begin{cor} \label{cor:FSMstable}
Let $A$ be a rich band-dominated operator, $c>0$ and $(A_n), (A_{n,c})$ as defined in \eqref{eq:An} and \eqref{eq:Anc}. Then:
$(A_n)$ is stable if and only if all limit operators of $\oplus A_{n,c}$ are invertible.
\end{cor}

So we are led to studying the limit operators of $\oplus A_{n,c}$. It is easy to see that each of them is again a stacked operator, say $\oplus B_n$. A detailed analysis of $\oplus A_{n,c}$ and its limit operators (see e.g. \cite{RochOnFS,LiFS,HabilMarko}) shows that the operators $B_n$ to be considered here are:
\begin{itemize}
\itemsep-1mm
\item[(a)] the operator $A$ itself,
\item[(b)] $c$ times the identity operator on $\Xb$,
\item[(c)] all limit operators of $A$,
\item[(d)] certain truncated limit operators of $A$, extended to $\Xb$ by $c$ times the identity, and
\item[(e)] shifts of all the operators above.
\end{itemize}

The invertibility of all limit operators of $\oplus A_{n,c}$ reduces to the invertibility\footnote{\label{foot:13}The uniform boundedness of all inverses $B_n^{-1}$ follows automatically from their existence, as can be seen by a slight modification of our $\oplus A_{n,c}$ construction: Assemble $A_1, A_2,...$ into one diagonal operator $D:=\sum_{k\in \Nb}V_{g_{k}}A_{k}V_{-g_{k}}+c(I-\sum_{k\in \Nb}V_{g_{k}}P_{k}V_{-g_{k}})$,
acting on $\Xb$ (not $\Xb'$), where the $g_{k}$ are chosen such that the sets $g_{k}+\{-k,...,k\}^{N}$ are pairwise disjoint, as in \eqref{ET} above. Then (see e.g. \cite{RochOnFS}) the set of all operators $B_{n}$ in the limit operators $\oplus B_{n}$ of  $\oplus A_{n,c}$ coincides with the set of all limit operators of $D$, so that, by Theorem \ref{TBdORichHS}, the inverses of all $B_n$ are uniformly bounded as soon as they all exist.} of all $B_n$ under consideration, which is of course handy since it brings us back to the $\Xb\to\Xb$ setting of the original operator $A$.
In terms of invertibility of all members $B_n$, there is a lot of redundancy in the list (a)--(e) since $cI$ is invertible, shifts do not change invertibility, and invertibility (even $\Pc$-Fredholmness) of $A$ implies that of all its limit operators. So it remains to look at points (a) and (d).
Without going into the details of (d), we will denote this remaining set $\{$(a),(d)$\}$ of operators by $S(A,c)$; in \cite{RochOnFS,LiFS,HabilMarko} it is called the {\sl stability spectrum of $A$}.
From Corollary \ref{cor:FSMstable} and the discussion above one gets that
\begin{equation} \label{eq:stab2}
(A_n)\ \textrm{is stable}\quad\iff\quad(A_{n,c})\ \textrm{is stable}\quad\iff\quad
\textrm{all elements of $S(A,c)$ are invertible.}
\end{equation}
\begin{ex} \label{ex:blockdiag}
Let $\Xb=l^{2}(\Zb^{1},\Cb)$, $\mu\in[0,1)$  and consider the operator $A$ induced by the block
diagonal matrix
\[
A\ =\ \diag(\cdots,B,B,1,B,B,\cdots),\quad\textrm{where\ }
B\ =\ \left(\begin{array}{cc}\mu&1\\1&\mu\end{array}\right)
\]
and the single $1$ entry is at position $(0,0)$ of $A$. Then $A$ is
invertible with
\[
A^{-1}=\diag(\cdots,B^{-1},B^{-1},1,B^{-1},B^{-1},\cdots),\quad\textrm{where\ }
B^{-1} =\ \frac{1}{\mu^2-1}\left(\begin{array}{cc}\mu&-1\\-1&\mu\end{array}\right)
\]
and its finite sections correspond to the finite $(2n+1) \times (2n+1)$ matrices
\[
A_{n}\ =\ \left\{
\begin{array}{cl}
\diag\left(B,\cdots,B,1,B,\cdots,B\right) & \textrm{if $n\ge 2$ is even},\\
\diag\left(\mu,B,\cdots,B,1,B,\cdots,B, \mu\right) & \textrm{if $n\ge 3$ is odd.}
\end{array}\right.
\]
If $\mu=0$ then the $A_{n}$ with odd $n$ are singular so that the sequence $(A_{n})$ is not stable.
If $\mu\in (0,1)$ then all $A_{n}$ are invertible, with
\[
\|A_{n}^{-1}\|\ =\ \left\{
\begin{array}{ll}
\|B^{-1}\|=(1-\mu)^{-1} & \textrm{if $n\ge 2$ is even,}\\
\max\left(\|B^{-1}\|,\mu^{-1}\right)=\max\left((1-\mu)^{-1},\mu^{-1}\right)=(\min(1-\mu,\mu))^{-1} & \textrm{if $n\ge 3$ is odd.}
\end{array}\right.
\]
So for $\mu\in (0,1)$ the sequence $(A_{n})$ is stable, where the limsup in \eqref{eq:stab} equals
$(\min(1-\mu,\mu))^{-1}$. This limsup is a limit if and only if $(\min(1-\mu,\mu))^{-1}=(1-\mu)^{-1}$, i.e. if $\mu\in [\frac12,1)$.

Fix $c\ge\|A\|=\|B\|$, e.g. $c:=2$. Then $A_{n,c}=\diag(\cdots,c,c,A_n,c,c,\cdots)$ and the stability spectrum $S(A,c)$ consists in this example of five operators. They
are $A$,
\[
\begin{array}{ll}
C=\diag\left(\cdots,c,c,c,B,B,B, \cdots \right) &
D=\diag\left(\cdots, B,B,B,c,c,c,\cdots \right),\\[1mm]
E=\diag\left(\cdots,c,c,c,\mu,B,B, \cdots \right),&
F=\diag\left(\cdots,B,B,\mu,c,c,c,\cdots \right).
\end{array}
\]
In case $\mu=0$ only $A, C$ and $D$ are invertible. In case $\mu\in (0,1)$, all five operators are invertible, where $\|A^{-1}\|=\|B^{-1}\|=\|C^{-1}\|=\|D^{-1}\|=(1-\mu)^{-1}$ and
$\|E^{-1}\|=\|F^{-1}\|=(\min(1-\mu,\mu))^{-1}$.
\end{ex}

This example suggests that the set $S(A,c)$ not only determines the stability of $(A_{n})$ via the invertibility of all members of $S(A,c)$, by \eqref{eq:stab2}, but also the answer to question (Q1) via the norms of those inverses. It also shows that the answer to question (Q2) is usually negative. Questions (Q3) and (Q4) are fairly straightforward once (Q1) and (Q2) are settled. As we said, in this paper we restrict ourselves to (Q1) and (Q3).
So let us turn back to the general setting.

\subparagraph{On question (Q1): What is $\limsup\|A_n^{-1}\|$?}
We start by noting that the elements of $S(A,c)$ are not just those operators from the list (a)--(e) whose invertibility implies that of all other operators on that list -- but they also have the largest inverses among (a)--(e), provided that $c$ is large enough.

\begin{prop} \label{prop:S(A)norms}
Let $A\in\Alpr$, $c\ge\|A\|$ and $(A_{n,c})$ as in \eqref{eq:Anc}. Then
\[
\max_{L\in\opsp(\oplus A_{n,c})}\|L^{-1}\| = \max_{S\in \{{\rm (a)-(e)}\}}\|S^{-1}\|  = \max_{S\in S(A,c)}\|S^{-1}\|.
\]
\end{prop}
\begin{proof}
From Theorem \ref{TEssNorm} we know that the LHS indeed exists as a maximum. As in the discussion following Corollary \ref{cor:FSMstable}, we note that each $L\in\opsp(\oplus A_{n,c})$ is of the form $L=\oplus B_n$, so that $L^{-1}=\oplus B_n^{-1}$ and the maximum of all $\|L^{-1}\|$ is the supremum of all $\|B_n^{-1}\|$. As in footnote \ref{foot:13}, using Theorem 8 of \cite{BigQuest}, one can see that also this supremum is attained as a maximum.
So the LHS equals the maximum of $\|S^{-1}\|$ with $S$ from the list (a)--(e).
It remains to show that this maximum is attained in items (a) or (d).

(a) vs.~(b): From $\|A\|\le c$ we get that $\|A^{-1}\|\ge\|A\|^{-1}\ge c^{-1}=\|(cI)^{-1}\|$.

(a) vs.~(c): Let $A$ be invertible and $A_h\in\opsp(A)$. Then $A_h$ is invertible and $(A_h)^{-1}=(A^{-1})_h$, by Proposition \ref{prop:limops}. By the same proposition, $\|A^{-1}\|\ge\|(A^{-1})_h\|=\|(A_h)^{-1}\|$.

(e): Clearly, taking shifts does not change the norm of the inverse.
\end{proof}

The maximum of $\|S^{-1}\|$ can be attained by (a), $S=A$, or by (d), a truncated limit operator of $A$. See Example \ref{ex:blockdiag} with $\mu\in(0,\frac12)$ for the latter, and replace $A$ by $A=\diag(\cdots,B,B,\frac{\mu}{2},B,B,\cdots)$, again with $\mu\in(0,\frac12)$, for the former.
Next we rewrite $\limsup\|B_n\|$ as the $\Pc'$-essential norm of the stacked operator $\oplus B_n$:
\begin{lem} \label{lem:limsup}
Consider a bounded sequence $(C_n)_{n\in\Zb}$ with $C_n:\im P_n\to\im P_n$ for $n\in\Nb$ and $C_n=0$ for $n\in\Zb\setminus\Nb$.
Now let $0\le d\le\inf_{n\in\Nb}\|C_n\|$ and $B_n:=C_n+dQ_n$.
Then $\oplus B_n\in\lb{\Xb',\Pc'}$ and
\[
\|\oplus B_n+\lc{\Xb',\Pc'}\| = \limsup_{n\to\infty}\|B_n\|  = \limsup_{n\to\infty}\|C_n\|.
\]
\end{lem}
\begin{proof}
By the construction of $B_n$, we have $Q_k'(\oplus B_n)P_m'=0=P_m'(\oplus B_n)Q_k'$ for all $m\in\Nb$ and $k\ge m$, so that $\oplus B_n\in\lb{\Xb',\Pc'}$, by \cite[Prop. 1.1.8]{LimOps}.
Using $\|B_n\|=\max(\|C_n\|,d)=\|C_n\|$ for all $n\in\Nb$, we derive the equality
\[
\|\oplus B_n+\lc{\Xb',\Pc'}\| = \lim_{m\to\infty}\|Q_{m}'(\oplus B_n)\| = \lim_{m\to\infty}(\sup_{n>m}\|B_{n}\|)= \limsup_{n\to\infty}\|B_n\|
= \limsup_{n\to\infty}\|C_n\|\]
from Proposition \ref{PEssNorm} and \eqref{eq:normsup}.
\end{proof}

Now we are ready to answer question (Q1):

\begin{thm} \label{thm:Q1}
Let $A\in\Alpr$, $c\ge\|A\|$ and let $(A_n)$ from \eqref{eq:An} be stable. Then
\[
\limsup_{n\to\infty}\|A_n^{-1}\| = \max_{S\in S(A,c)}\|S^{-1}\|.
\]
\end{thm}
\begin{proof}
Fix $n_0\in\Nb$ so that all $A_n$ and $A_{n,c}$ with $n\ge n_0$ are invertible. Then $\oplus B_n$ with $B_n=A_{n,c}^{-1}$ for $n\ge n_0$ and $B_n=c^{-1}I$ for $n<n_0$ is a $\Pc'$-regularizer for $\oplus A_{n,c}$.
From $c\ge\|A\|$ we get
\begin{align*}
\limsup_{n\to\infty}\|A_n^{-1}\|
&\ =\ \limsup_{n\to\infty}\|A_{n,c}^{-1}\|
&\text{(by \eqref{eq:normAnc})}\\
&\ =\ \limsup_{n\to\infty}\|B_n\| \ =\ \|\oplus B_n+\lc{\Xb',\Pc'}\| &\text{(by Lemma \ref{lem:limsup})}\\
&\ =\ \|\left[\oplus A_{n,c}+\lc{\Xb',\Pc'}\right]^{-1}\|
\ =\ \max_{L\in\opsp(\oplus A_{n,c})}\|L^{-1}\|
&\text{(by Corollary \ref{CEssNorm})}\\
&\ =\ \max_{S\in S(A,c)}\|S^{-1}\|
&\text{(by Proposition \ref{prop:S(A)norms})},
\end{align*}
which finishes the proof.
\end{proof}

\subparagraph{On question (Q3): The asymptotics of the condition numbers}
From $A_n\pto A$ together with \eqref{eq:liminf} and $\|A_n\|=\|P_nAP_n\|\le\|A\|$ we get
$
\|A\|\le\liminf\|A_n\|\le\limsup\|A_n\|\le\|A\|,
$
so that $\lim\|A_n\|$ exists and equals $\|A\|$. So the asymptotics of the condition numbers $\kappa(A_n)=\|A_n\|\|A_n^{-1}\|$ is essentially governed by the asymptotics of $\|A_n^{-1}\|$:

\begin{cor} \label{cor:Q3}
Let $A\in\Alpr$, $c\ge\|A\|$ and let $(A_n)$ from \eqref{eq:An} be stable. Then
\[
\limsup_{n\to\infty}\,\kappa(A_n) = \|A\|\cdot\max_{S\in S(A,c)}\|S^{-1}\|.
\]
\end{cor}

If $\limsup\|A_n^{-1}\|$ is a limit then also $\limsup\kappa(A_n)$ is a limit, but whether or when this happens is the subject of our question (Q2), which is addressed in \cite{HagLiSei:FSM}.
\medskip

We want to mention that versions of both results, Theorem \ref{thm:Q1} and Corollary \ref{cor:Q3}, are already contained in the literature: In the Hilbert space case they follow directly from \eqref{eq:stab2} by a $C^*$-algebra argument (as in footnote \ref{foot:C*alg}). \cite{MaSaSe} gives such results while even exceeding the setting of band-dominated operators considerably.
The general case  $\Xb=l^p(\Zb^N,X)$ is studied in \cite{SeSi4} and in Section 3.2 of \cite{SeDis}. While the results of \cite{SeDis,SeSi4} even apply to sequences $(A_n)$ in an algebra of finite section sequences, they put stronger constraints on the operator $A$ (the higher the dimension $N$, the stronger are the constraints). Our current approach shows how to avoid these constraints on $A$, and our separate paper \cite{HagLiSei:FSM} combines the benefits of the two approaches.



\newpage 
\noindent
\underline{\bf Authors:}\\[2mm]
Raffael Hagger \& Marko Lindner\\
Hamburg University of Technology (TUHH), Institute of Mathematics, 21073 Hamburg, Germany\\
\href{mailto:Raffael.Hagger@tuhh.de}{\tt Raffael.Hagger@tuhh.de},\quad
\href{mailto:Marko.Lindner@tuhh.de}{\tt Marko.Lindner@tuhh.de}
~\\[3mm]
Markus Seidel\\
University of Applied Sciences Zwickau, Dr.-Friedrichs-Ring 2a, 08056 Zwickau, Germany\\
\href{mailto:Markus.Seidel@fh-zwickau.de}{\tt Markus.Seidel@fh-zwickau.de}

\end{document}